\documentclass[11pt]{article} 

\usepackage{amsmath,amssymb,array,enumerate}

\def\r{\rightarrow}

\newcommand{\fdem}{\hspace*{\fill}~$\Box$\par\endtrivlist\unskip}

\newcommand{\N}{\mathbb{N}}     

\newcommand{\R}{\mathbb{R}}     
     
\newcommand{\C}{\mathbb{C}}

\newcommand{\X}{\mathbb{X}}

\renewcommand{\dim}{\mathop{\rm dim}}
\renewcommand{\ker}{\mathop{\rm Ker}}

\renewcommand{\r}{\mathop{\rightarrow}}

\newcommand{\cB}{\mbox{$\cal B$}}

\newcommand{\cD}{{\cal D}}
\newcommand{\cE}{\mbox{$\cal E$}}

\newcommand{\cL}{\mbox{$\cal L$}}

\newcommand{\cX}{\mbox{$\cal X$}}

\newcommand{\cV}{{\cal V}}

\newcommand{\Pc}{\widehat{P}}
\newcommand{\gc}{\widehat{\gamma}}

\newenvironment{proof}[1]{\textit{Proof#1.\,}}{\fdem}

\newtheorem{atheo}{Theorem}[section]
\newtheorem{adefi}[atheo]{Definition}
\newtheorem{alem}[atheo]{Lemma}
\newtheorem{arem}[atheo]{Remark}

\newtheorem{apro}[atheo]{Proposition}

\title{Approximating Markov chains and $V$-geometric ergodicity via weak perturbation theory}
\author{Lo\"ic HERV\'E and James LEDOUX \footnote{INSA de Rennes, F-35708, France \& IRMAR CNRS-UMR 6625, Rennes, F-35000, France; Universit\'e Europ\'eenne de Bretagne, France.
 Lo\"ic.Herve@insa-rennes.fr, James.Ledoux@insa-rennes.fr}
}

\begin{document}

\maketitle

\begin{abstract}
Let $P$ be a Markov kernel on a measurable space $\X$ and 
let $V:\X\r[1,+\infty)$.
This paper provides explicit connections between the $V$-geometric 
ergodicity of $P$ and that of finite-rank nonnegative sub-Markov kernels $\Pc_k$ approximating $P$. A special attention is paid to obtain an efficient way to specify the convergence rate for $P$ from that of $\Pc_k$ and conversely.  
Furthermore, explicit bounds are obtained for the total variation distance between 
the $P$-invariant probability measure and the $\Pc_k$-invariant positive measure. 
The proofs are based on the Keller-Liverani perturbation theorem which 
requires an accurate control of the essential spectral radius of $P$ on usual weighted supremum spaces. Such computable bounds are derived in terms of standard drift conditions. Our spectral procedure to estimate both the convergence rate and the invariant probability measure of $P$ is applied to truncation of discrete Markov kernels on $\X:=\N$.

\end{abstract}

\begin{center}


\end{center}

%=====================
%=====================
%=====================
\section{Introduction} \label{intro}
%====================
%=====================
%=====================
Throughout the paper $P$ is a Markov kernel on a measurable space $(\X,\cX)$, and $\{\Pc_k\}_{k\geq1}$ is a sequence of nonnegative sub-Markov kernels on $(\X,\cX)$. 
For any positive measure $\mu$ on $\X$ and any $\mu$-integrable function $f : \X\r\C$, $\mu(f)$ denotes the integral $\int fd\mu$. Let $V : \X\r[1,+\infty)$ be a measurable function such that $V(x_0)=1$ for some $x_0\in\X$. Let $(\cB_1,\|\cdot\|_1)$ denote the weighted-supremum Banach space 
$$\cB_1 := \big\{ \ f : \X\r\C, \text{ measurable }: \|f\|_1  := \sup_{x\in\X} |f(x)| V(x)^{-1} < \infty\ \big\}.$$ 
In the sequel, $PV/V$ and each $\Pc_k V/V$ are assumed to be bounded on $\X$, so that both $P$ and $\Pc_k$ are bounded linear operators on $\cB_1$. Moreover $P$ (resp.~every $\Pc_k$) is assumed to have an invariant probability measure $\pi$ (resp.~an invariant bounded positive measure $\widehat\pi_k$) on $(\X,\cX)$ such that $\pi(V)< \infty$ (resp. $\widehat\pi_k(V)< \infty$). We suppose that there exists a bounded measurable function $\phi_k:\X\r[0,+\infty)$ such that $\Pc_k\phi_k = \phi_k$ and $\widehat\pi_k(\phi_k) = 1$. 
Throughout the paper, every $\Pc_k$ is of finite rank on $\cB_1$ and the sequence $\{\Pc_k\}_{k\ge 1}$  converges to $P$ in the weak sense (\ref{C01}) below. Finally we assume that 
\begin{equation} \label{cond-phi-k}
\lim_{k\r+\infty} \pi(\phi_k) = 1  \quad \text{and} \quad\lim_{k\r+\infty} \widehat\pi_k(1_\X) = 1 . 
\end{equation}
These two conditions are necessary for the sequence $\{\widehat\pi_k\}_{k\ge 1}$ to converge to $\pi$ in total variation distance since $\pi(\phi_k) - 1 = (\pi - \widehat\pi_k)(\phi_k)$ and 
$\widehat\pi_k(1_\X) - 1 = (\widehat\pi_k-\pi)(1_\X)$. 

When $P$ is a Markov kernel on $\X:=\N$, a typical instance of sequence $\{\Pc_k\}_{k\ge 1}$  is obtained by considering the extended sub-Markov kernel  $\Pc_k$ derived from the linear augmentation (in the last column) of the $(k+1)\times (k+1)$ northwest corner truncation $P_k$ of $P$ (e.g. see \cite{Twe98}). Then $\widehat\pi_k$ is the (extended) probability measure on $\N$ derived from the $P_k$-invariant probability measure and $\phi_k = 1_{B_k}$ with $B_k := \{0,\ldots,k\}$. 

In this work, the connection between the $V$-geometrical ergodicity of $P$ and that of  $\Pc_k$ is investigated. These properties are defined as follows. 
\leftmargini 0em
{\it \begin{itemize}  
	\item[] \textbf{ 
$P$ (resp.~$\Pc_k$) is said to be $V$-geometrically ergodic} 
if there exist some rate $\rho\in(0,1)$ (resp.~$\rho_k\in(0,1)$) and constant $C>0$ (resp.~$C_k>0$) such that
\begin{gather} 
\forall n\geq 0,\ \sup_{f\in{\cal B}_1,\|f\|_1\leq 1} \|P^nf- \pi(f)1_\X\|_1 \leq C\, \rho^n \tag{$V$}\label{V-geo-cond} \\
respectively: \quad \forall n\geq 0,\ \sup_{f\in{\cal B}_1,\|f\|_1\leq 1}\|{\Pc_k}^{\;n} f - \widehat\pi_k(f)\phi_k\|_1 \leq C_k\, {\rho_k}^n. \qquad \qquad  \tag{$V_k$} \label{Vk-geo-cond} 
\end{gather}   
\end{itemize}}
\noindent Specifically, the two following issues are studied. 
\leftmargini 2.4em
\begin{enumerate}[(Q1)] {\it 
	\item \label{1} Suppose that $P$ is $V$-geometrically ergodic. 
\begin{enumerate}[(a)]
	\item Is $\Pc_k$ a $V$-geometrically ergodic kernel for $k$ large enough?
	\item When $(\rho,C)$ is known in \emph{(\ref{V-geo-cond})}, can we deduce explicit $(\rho_k,C_k)$ in \emph{(\ref{Vk-geo-cond})} from $(\rho,C)$?
	\item Does the total variation distance $\| \widehat\pi_k - \pi \|_{TV}$ go to $0$ when $k\r+\infty$, and can we obtain an explicit bound for   $\| \widehat\pi_k - \pi \|_{TV}$ when $(\rho,C)$ is known in~\emph{(\ref{V-geo-cond})}? 
\end{enumerate}
 \item \label{2} Suppose that $\Pc_k$ is $V$-geometrically ergodic for every $k$. 
		\begin{enumerate}[(a)] 
	\item  Is $P$ a $V$-geometrically ergodic kernel? 
	\item 	When $(\rho_k,C_k)$ is known in \emph{(\ref{Vk-geo-cond})}, can we deduce explicit $(\rho,C)$ in \emph{(\ref{V-geo-cond})} from $(\rho_k,C_k)$, and consequently obtain a bound for $\| \widehat\pi_k - \pi \|_{TV}$ using the last part of (Q\ref{1}$(c)$)? 
\end{enumerate}}
\end{enumerate}

A natural way to solve (Q\ref{1}) is to see $\Pc_k$ as a perturbed operator of $P$, and vice versa 
for (Q\ref{2}). The standard perturbation theory requires that $\{\Pc_k\}_{k\ge 1}$ converges to $P$ in operator norm on $\cB_1$. Unfortunately this condition may be very restrictive even if $\X$ is discrete (e.g.~see \cite{ShaStu00,FerHerLed11}): for instance, in our application to  truncation  of discrete kernels in  Section~\ref{sec-appli}, this condition never holds (see Remark~\ref{Annexe_Continuity}). Here we use the weak perturbation theory due to Keller and Liverani \cite{KelLiv99,Liv01} (see also \cite{Bal00})  which invokes the  weakened convergence property  
\begin{equation} \label{C01}
\|\Pc_k  - P\|_{0,1} : = \sup_{f\in{\cal B}_0,\|f\|_0\leq 1} \|\Pc_kf- Pf\|_1 \xrightarrow[k\r +\infty ]{} 0, 
\tag{$C_{0,1}$}
\end{equation}
where $\cB_0$ is the Banach space of bounded measurable $\C$-valued functions on $\X$ equipped with its usual norm $\|f\|_0:=\sup_{x\in\X}|f(x)|$. In the truncation context of Section~\ref{sec-appli} where $\X:=\N$, Condition (\ref{C01}) holds provided that $\lim_k V(k)= +\infty$. The price to pay for using~(\ref{C01}) is that two functional assumptions are needed. The first one involves the Doeblin-Fortet inequalities: such dual inequalities can be derived for $P$ and every $\Pc_k$ from 
Condition (\ref{drift-gene-cond-gene}) below. The second one requires to have an accurate bound of the essential spectral radius $r_{ess}(P)$ of $P$ acting on $\cB_1$. 

Issues (Q\ref{1}) and (Q\ref{2}) are solved in Sections~\ref{sec-V-to-Vk} and \ref{sec-Vk-to-V}. The question in (Q\ref{1})-(Q\ref{2}) concerning $\| \widehat\pi_k - \pi \|_{TV}$ is solved by  arguments inspired by \cite[Lem.~7.1]{Liv01} (see Proposition~\ref{pro-P-hatP-en3fois}).  The other questions in both (Q\ref{1})-(Q\ref{2}) are addressed using \cite{Liv01}. Theorems~\ref{pro-bv} and \ref{cor-Vk-to-V} provide a positive and explicit answer to the issues (Q\ref{1}) and (Q\ref{2}) under Conditions~(\ref{V-geo-cond}), (\ref{Vk-geo-cond}), (\ref{C01}) and the following uniform weak drift condition: 
\begin{equation} \label{drift-gene-cond-gene} 
\exists\, \delta\in(0,1),\ \exists L>0,\ \forall k\in\N^*\cup\{\infty\}, \quad  \Pc_kV \leq \delta V + L\, 1_{\X}  \tag{\text{UWD}}
\end{equation}
where by convention $\Pc_\infty := P$. The notion of essential spectral radius and the material on quasi-compactness used for bounding $r_{ess}(P)$ are reported in Section~\ref{sec-mino}. Our results are applied in Section~\ref{sec-appli} to truncation of discrete Markov kernels. Some useful theoretical complements on the weak perturbation theory are postponed to  Section~\ref{WP_truncation}. 

The estimates of Theorems~\ref{pro-bv} and \ref{cor-Vk-to-V} are all the more precise that the real number $\delta$ in (\ref{drift-gene-cond-gene}) and the bound for $r_{ess}(P)$ 
are accurate. The weak drift inequality used in (\ref{drift-gene-cond-gene}) is a simple and well-known condition introduced in  \cite{MeyTwe93} for studying $V$-geometric ergodicity.  
Managing $r_{ess}(P)$ is more delicate, even in discrete case. 
To that effect, it is important not to confuse $r_{ess}(P)$ with $\rho$ in (\ref{V-geo-cond}). 
Actually Inequality~(\ref{V-geo-cond}) gives\footnote{
(See e.g.~\cite{KonMey03} or apply Definition~\ref{def-q-c} with $H:=\{f\in\cB_1 : \pi(f)=0\}$.)}
$r_{ess}(P)\leq \rho$, but this bound may be very inaccurate since 
 there exist Markov kernels $P$ such that $\rho-r_{ess}(P)$ is close to one (see Definition~\ref{def-q-c}). Moreover note that no precise rate $\rho$ in (\ref{V-geo-cond}) is known a priori in Issue (Q\ref{2}). Accordingly, our study requires to obtain a specific control of the essential spectral radius $r_{ess}(P)$ of a general Markov kernel $P$ acting on $\cB_1$. This study,  which has its own interest, is presented in Section~\ref{sec-mino} where the two following results are presented. 
\leftmargini 3.8em
\begin{enumerate}[(i)] {\it 
	\item If $P$ satisfies  some classical drift/minorization conditions, then $r_{ess}(P)$ can be bounded in terms of the constants of these drift/minorization conditions (see Theorem~\ref{main}). 
	\item If $PV \leq \delta V + L\, 1_{\X}$ for some $\delta\in(0,1)$ and $L>0$, and if $P : \cB_0\r\cB_1$ is compact, then 
	$r_{ess}(P)\leq \delta$ (see Proposition~\ref{pro-qc-bis}).}
\end{enumerate}
To the best of our knowledge, the general result~$(i)$ is not known in the literature. Result~$(ii)$ is a simplified version of \cite[Th.~3.11]{Wu04}. For convenience  we present a direct and short proof using \cite[Cor.~1]{Hen93}. The bound $r_{ess}(P) \leq \delta$ obtained in $(ii)$ is more precise than that obtained in $(i)$ (and sometimes is optimal). The compactness property in $(ii)$ must not be confused with that of $P:\cB_1\r\cB_1$ or $P:\cB_0\r\cB_0$, which are much stronger conditions. If $P$ is an infinite matrix ($\X:=\N$), then $P:\cB_0\r\cB_1$ is compact when $\lim_{k}V(k)=+\infty$, while in general $P$ is compact neither on $\cB_1$ nor on $\cB_0$.  

Let us give a brief review of previous works related to Issues (Q\ref{1}) or (Q\ref{2}). 
Various probabilistic methods  have been developed to derive explicit rate and constant $(\rho,C$) in Inequality~(\ref{V-geo-cond}) from the constants of drift conditions (see \cite{MeyTwe94,LunTwe96,Bax05} and the references therein). To the best of our knowledge, these methods, which are not concerned with approximation issues, provide a computable rate $\rho$ which is often unsatisfactory, except for reversible or stochastically monotone~$P$. Convergence of $\{\widehat\pi_k\}_{k\geq1}$ to $\pi$ has been studied in \cite{Twe98} for truncation approximations of $V$-geometrically ergodic Markov kernels with discrete $\X$. 
Specifically it is proved that 
$\sup_{|f|\leq V}|\widehat\pi_k(f) - \pi(f)|$ goes to $0$ when $k\r+\infty$. In particular $\|\widehat\pi_k - \pi\|_{TV}$ goes to $0$. In the special case when $P$ is stochastically monotone,  
the following rate of convergence is obtained \cite[Th. 4.2,(46)]{Twe98} 
\begin{equation} \label{ineg-Tweedie} 
 \forall k\in\N^*,\quad \|\widehat\pi_k - \pi\|_{TV} 
  = O\left(\frac{\ln V(k)}{V(k)}\right).
\end{equation}
In this estimate, explicit constants only depending on constants involved in drift/minorization conditions are also provided. A similar result has been obtained for polynomially ergodic Markov chains in \cite{Liu10}. A bound of type (\ref{ineg-Tweedie}) is provided by Theorems~\ref{pro-bv} and \ref{cor-Vk-to-V} under Conditions~(\ref{V-geo-cond}), (\ref{Vk-geo-cond}),  (\ref{C01}) and (\ref{drift-gene-cond-gene}). This result is new even for discrete truncation approximation since the stochastic monotonicity is not required in our work.  
Questions related to Issue~(Q\ref{2}) have been investigated in \cite{Ros96} for discrete reversible Markov kernels. The main result in \cite[Cor.~3]{Ros96} is that $\|\mu P^n - \pi\|_{TV} \leq c\, \beta^n$ for some explicit positive constant $c$,  provided that $\liminf\beta_k \leq \beta$, where $\beta_k$ denotes the so-called second eigenvalue of the $k$-th truncated kernels augmented on the diagonal. However, except in special cases (see \cite{Ros96}), finding a nontrivial upper  bound $\beta$ of $\liminf\beta_k$ is a difficult problem. Note that our Theorem~\ref{cor-Vk-to-V} only needs to compute the second eigenvalue of $\Pc_k$ for some suitable  $k$.  Finally mention that perturbations bounds have been obtained in \cite{Mit05} by a different way for uniformly ergodic Markov chains (i.e. $V\equiv 1$ in (\ref{V-geo-cond})).

The weak perturbation results of \cite{KelLiv99,Liv01} have been fully used in the framework of dynamical systems (e.g.~see \cite{Bah06,DemLiv08, BahBos10, BahBos11, Zhe10}). There, Markov kernels and their invariant probability measure are replaced by Perron-Frobenius operators and their so-called SRB measure. In the context of  $V$-geometrically ergodic Markov kernels, the Keller-Liverani theorem has been already used in \cite{FerHerLed11} to study general perturbation issues. 
When applied to our context, \cite[Th.~1]{FerHerLed11} gives a positive answer to (Q\ref{1})$(a)$ and the convergence of $\| \widehat\pi_k - \pi \|_{TV}$ to $0$.  
The others questions in (Q\ref{1})-(Q\ref{2}) are not addressed in \cite{FerHerLed11}. In the present work, the duality arguments introduced in \cite{FerHerLed11} are applied together with the explicit bounds of \cite{Liv01}. 

%==============================================
%==============================================
%==============================================
\section{Notations and preliminary results} \label{sect-tronc-tres-gene}
%==============================================
%==============================================
%==============================================

Let us first introduce notations  used in the paper. Let $\lfloor \cdot  \rfloor$ denote the integer part function. 
We denote by $(\cL(\cB_0,\cB_1),\|\cdot\|_{0,1})$ the space of all the bounded linear maps from $\cB_0$ to $\cB_1$, equipped with its usual norm: 
$$\|T\|_{0,1} = \sup\big\{\|Tf\|_1,\, f\in\cB_0,\, \|f\|_0 \leq 1\big\}.$$ 
We write $\cL(\cB_1)$ for $\cL(\cB_1,\cB_1)$ and $\|T\|_1$ for $\|T\|_{1,1}$. Let $(\cB_1',\|\cdot\|_1)$ be the dual space of $\cB_1$. Note that we make a slight abuse of notation in writing again $\|\cdot\|_1$ for the operator and dual norms on $\cB_1$. Recall that the total variation distance between $\widehat\pi_k$ and $\pi$ is defined by 
\begin{equation*} 
\| \widehat\pi_k - \pi \|_{TV} := \sup_{\|f\|_0\leq1} |\widehat\pi_k(f) - \pi(f)| .
\end{equation*}

The next proposition is relevant to estimate $\| \widehat\pi_k - \pi \|_{TV}$. 
\begin{apro} \label{pro-P-hatP-en3fois}
Assume that Condition~\emph{(\ref{drift-gene-cond-gene})} holds. Set $\Delta_k := \|\Pc_k  - P\|_{0,1}$ and 
\begin{equation} \label{As-bounded}
A := 1 +\frac{L}{1-\delta}.
\end{equation}
\begin{enumerate}[(a)]
	\item If, for some $k\geq1$, $\Pc_k$ is $V$-geometrically ergodic with rate and constant $(\rho_k,C_k)$ in \emph{(\ref{Vk-geo-cond})}, then
\begin{equation} \label{rem-spec-n-k-bis}
\| \widehat\pi_k - \pi \|_{TV}\  \leq \  \widehat\pi_k(1_\X)\, |1-\pi(\phi_k)| + \frac{L}{1-\delta}\bigg(\frac{2C_k}{\rho_k} + \frac{A}{\ln({\rho_k}^{\, -1})}\, |\ln \Delta_k|\bigg)\Delta_k  .
\end{equation}
   \item If $P$ is $V$-geometrically ergodic with rate and constant $(\rho,C)$ in \emph{(\ref{V-geo-cond})}, 
then 
\begin{equation} \label{rem-spec-n-k-ter}
\forall n\geq 1,\quad \| \widehat\pi_n - \pi \|_{TV}\  \leq\  |1-\widehat\pi_n(1_\X)| + \frac{L}{1-\delta}\bigg(\frac{2C}{\rho} + \frac{A}{\ln {\rho}^{-1}}\, |\ln \Delta_n|\bigg)\Delta_n. 
\end{equation}
\end{enumerate}
\end{apro}
From Assumption~(\ref{cond-phi-k}), we know that the first term in the right-hand side of both (\ref{rem-spec-n-k-bis}) and (\ref{rem-spec-n-k-ter}) converges to $0$. When applied to truncation of discrete kernels in Section~\ref{sec-appli}, the first term in the right-hand side of (\ref{rem-spec-n-k-bis}) (resp.~of (\ref{rem-spec-n-k-ter})) is $O(\Delta_k)$ (resp.~zero).

 Inequality~(\ref{rem-spec-n-k-bis}) is interesting since rate and constant $(\rho_k,C_k)$ in (\ref{Vk-geo-cond}) are expected to be computable. However further properties on the $\rho_k$'s and the $C_k$'s are required to deduce $\lim_k\| \widehat\pi_k - \pi \|_{TV} = 0$ from (\ref{rem-spec-n-k-bis}). They are derived from (\ref{V-geo-cond}) in Theorem~\ref{pro-bv}. By contrast, Inequality~(\ref{rem-spec-n-k-ter}) gives $\lim_n \| \widehat\pi_n - \pi \|_{TV} = 0$ when $\|\Pc_n  - P\|_{0,1}\r0$. But the bound $\| \widehat\pi_n - \pi \|_{TV} = O(|\ln \Delta_n|\Delta_n)$ provided by (\ref{rem-spec-n-k-ter}) is only computable when some explicit rate and constant $(\rho,C)$ are known in (\ref{V-geo-cond}). Such explicit  $(\rho,C)$ is derived from (\ref{Vk-geo-cond}) in Theorem~\ref{cor-Vk-to-V}. 

\noindent \begin{proof}{ of Proposition~\ref{pro-P-hatP-en3fois}}
In a first step, we prove that we have with $A_n := \max_{0\leq j \leq n-1} \|P^j\|_1$
\begin{equation}
\| \widehat\pi_k - \pi \|_{TV} \leq  \widehat\pi_k(1_\X)\, |1-\pi(\phi_k)| +   \frac{L}{1-\delta}\big(2C_k\, {\rho_k}^{ n} + n A_n \Delta_k\big). \label{rem-spec-n-k-first} 
\end{equation} 
Observe that, for any $n\geq0$, we can write  
\begin{equation*} 
\|P^n - {\Pc_k}^{\ n}\|_{0,1} \leq n A_n \|P - {\Pc_k}\|_{0,1}.
\end{equation*}
This inequality follows from $\|\Pc_k\|_0\leq1$ and an easy induction based on 
$$P^{n} - {\Pc_k}^{\ n} = P^{n-1}(P - {\Pc_k}) + (P^{n-1} - {\Pc_k}^{\ n-1}){\Pc_k}.$$
Using the triangle inequality, we obtain for any $f\in\cB_0$ such that $\|f\|_0\leq1$
\begin{eqnarray*}
|\widehat\pi_k(f) - \pi(f)|\ & = & \big|\widehat\pi_k(\Pc_k^{\;n}f) - \pi(P^nf)\big| \\ 
& \leq & \big| (\widehat\pi_k-\pi)(\Pc_k^{\;n}f)\big| + \big| \pi\big(\Pc_k^{\;n}f  - P^nf\big) \big| \\ 
& \leq &  \widehat\pi_k(|f|)\big|(\widehat\pi_k-\pi)(\phi_k)\big| +  \big|(\widehat\pi_k-\pi)(\Pc_k^{\;n}f-\widehat\pi_k(f)\phi_k)\big| + \|\pi\|_1\, \|{\Pc_k}^{\ n} - P^n \|_{0,1}  \\ 
& \leq &  \widehat\pi_k(1_\X)\, |\pi(\phi_k)-1| + \|\widehat\pi_k-\pi\|_1 \, \|\Pc_k^{\;n}f-\widehat\pi_k(f)\phi_k\|_1 + \pi(V)\, n A_n \Delta_k \\
& \leq & \widehat\pi_k(1_\X)\, |\pi(\phi_k)-1| + \big(\widehat\pi_k(V) + \pi(V)\big)\, C_k\, \rho_k^{\, n}\, \|f\|_1 + \pi(V)\, n A_n \Delta_k. 
\end{eqnarray*}
Note that  
\begin{equation} \label{maj-pi-V}
\max\big(\pi(V),\widehat\pi_k(V)\big) \leq L/(1-\delta)
\end{equation} 
from~(\ref{drift-gene-cond-gene}) since $\pi$ (resp.~$\widehat\pi_k$) is $P$-invariant (resp.~$\Pc_k$-invariant). Inequality (\ref{rem-spec-n-k-first}) then follows from $\|f\|_1 \leq \|f\|_0$. Now observe that $A_n \leq A$ from $P^j V \leq \delta^j V +L \sum_{i=0}^{j-1} \delta^i \leq A\, V$.  
Setting $n := \max\big(0\, ,\, \lfloor (\ln \rho_k)^{-1}\ln\Delta_k  \rfloor\big)$ in (\ref{rem-spec-n-k-first}) allows us to derive  Inequality~(\ref{rem-spec-n-k-bis}). 

The proof of (\ref{rem-spec-n-k-ter}) is similar by exchanging the role of $P$ and $\Pc_k$. 
\end{proof}

Now we prove that Condition~(\ref{drift-gene-cond-gene}) of Introduction provides some uniform dual Doeblin-Fortet inequalities for the family $\{P, \Pc_k, k\ge 1\}$. For any $Q\in\cL(\cB_1)$, we denote by $Q'$ its adjoint operator. 
Define the following auxiliary semi-norm on $\cB_1'$: 
\begin{equation} \label{semi_norm_0}
\forall f'\in\cB_1',\quad \|f'\|_0 := \sup\big\{|f'(f)|,\, f\in\cB_0, \|f\|_0\leq 1\big\}.
\end{equation} 
\begin{alem} \label{lem-D-F} Let $Q$ be any non-negative linear operator on the space $\cB_1$ satisfying 
$$\exists\,  \delta\in(0,1),\ \exists L>0, \quad Q V \leq \delta V + L\, 1_{\X}.$$ 
Then: 
$\forall f'\in\cB_1', \ \|{Q\,'} f'\|_1 \leq \delta \|f'\|_1 + L\|f'\|_0.$
\end{alem}
For a Markov kernel $Q$, the proof of this lemma is given in \cite{FerHerLed11}. Since only the non-negativity of the operator $Q$ plays a role in this proof, the details are omitted. From Lemma~\ref{lem-D-F} we obtain the following statement.
\begin{alem} \label{D-F-uniform} If Condition~\emph{(\ref{drift-gene-cond-gene})} holds with parameters $(\delta,L)$, then the kernels $Q:=P$ and $Q:=\Pc_k$ for $k\ge 1$ 
satisfy  the following uniform Doeblin-Fortet inequality on $\cB_1'$: 
\begin{equation*} 
\forall f'\in\cB_1', \quad \|Q'f'\|_{1} \leq \delta\|f'\|_{1} + L \|f'\|_0. 
\end{equation*} 
\end{alem}

%======================
%=====================
\section{From $P$ to $\Pc_k$: solution to (Q\ref{1})} \label{sec-V-to-Vk} 
%======================
%=====================
%======================
For any $(a,\theta)\in\C\times(0,+\infty)$, let us define $D(a,\theta):=\{z\in\C : |z-a| < \theta\}$ and $\overline{D}(a,\theta):=\{z\in\C : 
|z-a| \le \theta\}$. For any $T\in\cL(\cB_1)$, the spectrum of $T$ is denoted by $\sigma(T)$, and for any $(r,\vartheta)\in(0,1)^2$, we introduce the following subsets of the complex plan   
\begin{equation*}  
\cV(r,\vartheta,T) := \big\{ z\in\C : |z|\le r  \text{ or }\,  d(z,\sigma(T)) \leq \vartheta\big\} \quad \text{and} \quad \cV(r,\vartheta,T)^c := \C\setminus \cV(r,\vartheta,T),
\end{equation*}
where $d(z,\sigma(T)) := \inf\{|z-\lambda|,\, \lambda\in\sigma(T)\}$. Below $P$ is assumed to be $V$-geometrically ergodic. In particular $P$ is quasi-compact on $\cB_1$, that is: $r_{ess}(P) < 1$ (see Section~\ref{sec-mino}). Under the additional Condition~(\ref{drift-gene-cond-gene}), we set 
\begin{equation} \label{def-hat-alpha}
\hat\alpha := \max\big(r_{ess}(P),\delta\big). 
\end{equation} 
Lemma~\ref{D-F-uniform} ensures that $Q:=P$ and $Q:=\Pc_k$ for $k\ge 1$ 
satisfy  the following Doeblin-Fortet inequalities on $\cB_1'$ (use the fact  that $P$ and $\Pc_k$ are contraction on $\cB_0$): 
\begin{equation} \label{doeb-fortet-dual-itere}
\forall n\in\N^*,\ \forall f'\in\cB_1', \quad \|{Q'}^n f'\|_{1} \leq {\hat\alpha}^n\|f'\|_{1} + B \|f'\|_0 \qquad \text{with}\quad B:= \frac{L}{1-\hat\alpha}. 
\end{equation} 
For any $r>\hat\alpha$ and $\vartheta>0$, we set:  
\begin{subequations}
\begin{gather}
H \equiv H(r,\vartheta,P) := \sup_{z\in \cV(r,\vartheta,P)^c} \|(zI-P)^{-1}\|_1  \label{H-r-vartheta-stat-enonce}\\ 
n_1 \equiv n_1(r) := \left\lfloor \frac{\ln2}{\ln(r/\hat\alpha)}\right\rfloor +1 \quad \quad 
n_2 \equiv n_2(r,\vartheta,P) := \bigg\lfloor \frac{\ln\big(8B(B+3)\, r^{-n_1}H\big)}{\ln(r/\hat\alpha)}\bigg\rfloor + 1 \label{n1-n2} \nonumber \\
\varepsilon_1 \equiv \varepsilon_1(r,\vartheta,P) := \frac{r^{n_1+n_2}}{8B\big(H B +(1-r)^{-1}\big)}. \label{def-eps-1}
\end{gather}
\end{subequations}
\begin{atheo} \label{pro-bv} 
Let $P$ be a $V$-geometrically ergodic Markov kernel with rate $\rho\in(0,1)$ in \emph{(\ref{V-geo-cond})}. Assume that Conditions~\emph{(\ref{C01})} and \emph{(\ref{drift-gene-cond-gene})} hold. Let $(r,\vartheta)\in(0,1)^2$ be such that 
\begin{equation} \label{cor-cond-r-var}
\max(\hat\alpha,\rho) + \vartheta < r < 1-\vartheta.
\end{equation}
Then, for any $k\in\N^*$ such that the convergence condition \emph{(\ref{C01})} holds at rate $\varepsilon_1$ that is   
\begin{equation} \label{cond-cont-faible}
\|\Pc_k  - P\|_{0,1} \leq \varepsilon_1, \tag{$\cE_{0,1}$} 
\end{equation}
and such that \emph{(\ref{Vk-geo-cond})} holds with some rate $\rho_k$ satisfying $\rho_k < 1-\vartheta,$
the following inequalities are valid. 
\leftmargini 1.3em
\begin{enumerate} \setlength{\itemsep}{-4mm}
  \item Inequality~\emph{(\ref{Vk-geo-cond})} holds with rate $r$, more precisely 
\begin{equation} \label{unif-bound-rate}
\forall n\geq 0, \quad  \|{\Pc_k}^{\;n}  - \widehat\pi_k(\cdot)\phi_k\|_1 \leq  
 c\, r^{n+1}  \quad \text{ with } c\equiv c(r,\vartheta,P)  :=   
\frac{4(B+1)}{r^{n_1}(1-r)} + \frac{1}{2\, \varepsilon_1}.	
   \end{equation}
   \item We have, with $A$ defined in \emph{(\ref{As-bounded})} and $\Delta_k := \|\Pc_k  - P\|_{0,1}$,
\begin{equation} \label{th-TV-ineg} 
\| \widehat\pi_k - \pi \|_{TV} \ \leq\ \widehat\pi_k(1_\X)\, |1-\pi(\phi_k)| + 
\frac{L}{1-\delta}\bigg(2c + \frac{A}{\ln r^{-1}}\, |\ln \Delta_k|\bigg)\Delta_k. 
\end{equation}
\end{enumerate}
\end{atheo}
Under Conditions~(\ref{V-geo-cond}), (\ref{C01}) and (\ref{drift-gene-cond-gene}), the conclusions (\ref{unif-bound-rate})-(\ref{th-TV-ineg}) hold true for $k$ large enough, so that Issue~(Q\ref{1}) is solved. Indeed, the fact that Condition~(\ref{cond-cont-faible}) is fulfilled for $k$ large enough follows from (\ref{C01}). That (\ref{Vk-geo-cond}) holds with some rate $\rho_k < 1-\vartheta$ for $k$ large enough is more difficult to establish. This follows from Proposition~\ref{cor-bv-app} which ensures that a sufficient condition for such a property to hold is that $k\geq k_0$, with $k_0$ defined in (\ref{con-k0-r-var}).  However in practice we do not need to use $k_0$ since Property~(\ref{Vk-geo-cond}) with $\rho_k < 1-\vartheta$ may be fulfilled for $k<<k_0$. This explains why $k_0$ is not introduced in Theorem~\ref{pro-bv}. Actually, under Conditions~(\ref{V-geo-cond}), (\ref{C01}) and (\ref{drift-gene-cond-gene}), the results \cite[Prop.~3.1]{Liv01} and the spectral rank-stability property \cite[Cor.~3.1]{Liv01} directly provide Property~(\ref{unif-bound-rate}) when  Condition~(\ref{cond-cont-faible}) is replaced with $\|\Pc_k  - P\|_{0,1} \leq \varepsilon_0$ with $\varepsilon_0 \equiv \varepsilon_0(r,\vartheta,P)$ defined in (\ref{def-eps-0-app}). In this case the assumption (\ref{Vk-geo-cond}) with $\rho_k < 1-\vartheta$ can be dropped in Theorem~\ref{cor-Vk-to-V}. When $\varepsilon_0$ is strictly less than $\varepsilon_1$ given in (\ref{def-eps-1}), using  Conditions~(\ref{cond-cont-faible}) and $\rho_k < 1-\vartheta$ invokes smaller $k$. 

In fact Theorem~\ref{pro-bv} is only relevant  when the rate $\rho$ and the bound $C$ in (\ref{V-geo-cond}) are  known. That $\rho$ must be known is necessary to choose $(r,\vartheta)$ in (\ref{cor-cond-r-var}). 
Constant $C$ must be known for an effective computation of an upper bound of $H\equiv H(r,\vartheta,P)$ in (\ref{H-r-vartheta-stat-enonce}) (see Remark~\ref{sub-rate-H}). Note that $H$ is involved in the definition of $\varepsilon_1(r,\vartheta,P)$, thus in the computation of the bound $c(r,\vartheta,P)$ in~(\ref{unif-bound-rate}). Furthermore note that, although Inequality~(\ref{th-TV-ineg}) is interesting from a theoretical point of view since it provides $\lim_k \| \widehat\pi_k - \pi \|_{TV}=0$, the bound in (\ref{th-TV-ineg}) is only computable when some explicit rate and constant in (\ref{V-geo-cond}) are known.  Another way for obtaining a computable bound for $\|\widehat\pi_k - \pi \|_{TV}$ is presented in Theorem~\ref{cor-Vk-to-V}
\begin{arem}  \label{sub-rate-H}
When $P$ is $V$-geometrically ergodic with some rate and constant $(\rho,C)$ in (\ref{V-geo-cond}), we have for any $(r,\vartheta)\in(0,1)^2$ such that $\rho+\vartheta < r < 1-\vartheta$:  
\begin{eqnarray*}
H(r,\vartheta,P) =  \sup_{z\in \overline{D}(0,r)^c \cap \overline{D}(1,\vartheta)^c} \|(zI-P)^{-1}\|_1 &\leq &  \frac{\pi(V)}{\vartheta}+\frac{C}{r-\rho} \ \le\  \frac{\pi(V)+C}{\vartheta}.
\end{eqnarray*}
These properties are derived by writing $g\in\cB_1$ as $g= \pi(g)1_\X + h$ with $h:=g - \pi(g)1_\X$. Using \emph{(\ref{V-geo-cond})} first implies that $(zI-P)^{-1}$ is well-defined in $\cL(\cB_1)$ for any $z\in\C$ such that $|z| > \rho$ and $z\neq1$, with: $(zI-P)^{-1}(g) = \pi(g)(zI-P)^{-1}1_\X + (zI-P)^{-1}(h)$. The first term is $(\pi(g)/(z-1))1_\X$. The second one can be bounded in norm $\|\cdot\|_1$ by using Neumann series. Since $\pi$ may be unknown, note that $\pi(V) \leq L/(1-\delta)$ under~\emph{(\ref{drift-gene-cond-gene})} (see \emph{(\ref{maj-pi-V})}). 
\end{arem}

Using Conditions~(\ref{cond-cont-faible}) and (\ref{doeb-fortet-dual-itere}) together with the condition $\rho_k < 1 - \vartheta$ allow us to derive Theorem~\ref{pro-bv}  from the first part of \cite[Prop.~3.1]{Liv01} as follows. 

\noindent\begin{proof}{} 
For any $T\in\cL(\cB_1)$, recall that $\|T'\|_1 = \|T\|_1$ and $\sigma(T)=\sigma(T')$. 
Given any  real numbers $r>\hat\alpha$ and $\vartheta>0$, we use the quantities introduced in (\ref{H-r-vartheta-stat-enonce})-(\ref{def-eps-1}). Note that $H$ in (\ref{H-r-vartheta-stat-enonce}) may be equivalently defined with $P'$ in place of $P$ since the resolvents of $P$ and $P'$ have the same norm in $\cL(\cB_1)$ and $\cL(\cB_1')$ respectively and $\cV(r,\vartheta,P)= \cV(r,\vartheta,P')$. 

The first assertion of Theorem~\ref{pro-bv} is proved in two steps. In a first step, using (\ref{V-geo-cond}), (\ref{doeb-fortet-dual-itere}) and (\ref{cond-cont-faible}) for some $k\in\N^*$, we show that we have with $c(r,\vartheta,P)$ given in (\ref{unif-bound-rate})
\begin{equation} \label{th-implic1}
\sigma(\Pc_k)\subset \cV(r,\vartheta,P)
\quad \text{ and } \quad  \sup_{z\in \cV(r,\vartheta,P)^c} \|(zI- \Pc_k)^{-1}\|_1 \leq 
c(r,\vartheta,P). 
\end{equation}
Clearly (\ref{th-implic1}) will hold true if we establish that the following properties are valid: 
\begin{gather} 
\sigma({\Pc_k}\,')\subset \cV(r,\vartheta,P) \quad \text{ and } \quad \sup_{z\in \cV(r,\vartheta,P)^c} 
\|(zI- {\Pc_k}\,')^{-1}\|_1  \leq  c(r,\vartheta,P). \label{implic1-proof-dual} 
\end{gather}
In fact we prove that (\ref{implic1-proof-dual}), so (\ref{th-implic1}), hold for any $\vartheta > 0$ and $r>\hat\alpha$. To that effect, the Keller-Liverani perturbation theorem \cite{KelLiv99} is applied to the adjoint operators of $P$ and $\Pc_k$ acting on $\cB_1'$. The auxiliary semi-norm $\|\cdot\|_0$ on $\cB_1'$ has been introduced in (\ref{semi_norm_0}).  We have that $\forall  f'\in\cB_1',\ \|f'\|_0\leq \|f'\|_1$. Moreover we have for any $f'\in\cB_1'$ 
$$\|P'f'\|_0 \leq \|f'\|_0 \quad \text{ and } \quad \forall k\geq 1,\ \|{\Pc_k}\,'f'\|_0 \leq \|f'\|_0.$$
Indeed, for $K:=P$ and $K:={\Pc_k}$, we obtain $\|Kf\|_0 \leq \|f\|_0$ from the positivity of $K$ and $K1_\X\leq 1_\X$, so that we have for any $f'\in\cB_1'$ and $f\in\cB_0$, $\|f\|_0\leq 1$: 
$$|(K'f')(f)| = |f'(Kf)| \leq \|f'\|_0\, \|Kf\|_0 \leq \|f'\|_0.$$ 
Next, using (\ref{drift-gene-cond-gene}), we know from Lemma~\ref{D-F-uniform} that $P'$ and ${\Pc_k}\,'$ (for every $k\geq 1$) 
satisfy the uniform Doeblin-Fortet inequalities (\ref{doeb-fortet-dual-itere}) on $\cB_1'$. Moreover we obtain by duality and from Inequality~(\ref{cond-cont-faible}) 
$$\|{\Pc_k}\,'  - P'\|_{1,0}
:= \sup_{f'\in{\cal B}_1',\|f'\|_{{\cal B}_1'}\leq 1} \|{\Pc_k}\,'f'  - P'f'\|_0 = \|{\Pc_k}  - P\|_{0,1} \leq \varepsilon_1.$$ 
Finally we know that $P'$ is quasi-compact on $\cB_1'$ with essential spectral radius less than $\hat\alpha$, and that the essential spectral radius of ${\Pc_k}\,'$ is zero since ${\Pc_k}$ is of finite rank. The previous facts  and the first part of \cite[Prop.~3.1]{Liv01} then give (\ref{implic1-proof-dual}). 

In a second step, we prove that Inequality~(\ref{unif-bound-rate}) holds under the assumptions of Theorem~\ref{pro-bv}.  Since $P$ is $V$-geometrically ergodic, we have $\sigma(P)\subset D(0,\rho) \cup \{1\}$. Thus, it follows from (\ref{cor-cond-r-var}) and the first inclusion in (\ref{th-implic1}) that $\sigma(\Pc_k)\subset \overline{D}(0,r)\cup \overline{D}(1,\vartheta)$. Moreover, since $\Pc_k$ is assumed to be $V$-geometrically ergodic with rate $\rho_k < 1-\vartheta$, we obtain that 
\begin{equation*}  
\sigma(\Pc_k)\subset \overline{D}(0,r)\cup \{1\}
\end{equation*}
and that $(1/2i\pi)\oint_{|z-1|=\vartheta} (zI-\Pc_k)^{-1}\, dz = \widehat\pi_k(\cdot)\, \phi_k$ from standard spectral theory. Thus we have for any $\kappa\in(r,1-\vartheta)$
\begin{eqnarray*}
\forall n\ge 1, \quad{\Pc_k}^{\;n} & = & \frac{1}{2i\pi} \oint_{|z-1|=\vartheta} (zI-\Pc_k)^{-1}\, dz \ + \  \frac{1}{2i\pi} \oint_{|z|=\kappa} z^n\, (zI-\Pc_k)^{-1}\, dz \\
& = & \widehat\pi_k(\cdot)\phi_k \ + \  \frac{1}{2i\pi} \oint_{|z|=\kappa} z^n\, (zI-\Pc_k)^{-1}\, dz. 
\end{eqnarray*}
Finally it follows from (\ref{th-implic1}) that 
$\|\Pc_k^{\;n} -  \widehat\pi_k(\cdot)\phi_k\|_1 \leq c(r,\vartheta,P)\, {\kappa}^{\,n+1}$. Since $\kappa\in(r,1-\vartheta)$ is arbitrary, this inequality holds with $r$ in place of $\kappa$, and it gives (\ref{unif-bound-rate}). 

Using (\ref{unif-bound-rate}), the second assertion of Theorem~\ref{pro-bv} follows from Inequality (\ref{rem-spec-n-k-bis}) of Proposition~\ref{pro-P-hatP-en3fois}. The proof of Theorem~\ref{pro-bv} is complete. 
\end{proof}

%============================================
%============================================
\section{From $\Pc_k$ to $P$: solution to (Q\ref{2})} \label{sec-Vk-to-V} 
%============================================
%============================================
Let $\rho_V(P)$ denote the spectral gap of $P$, that is the infimum bound of the rates $\rho$ in Inequality~(\ref{V-geo-cond}). Recall that $\rho_V(P)$ is unknown in general and that, even when it is known, finding an explicit constant $C$  associated with $\rho>\rho_V(P)$ in  (\ref{V-geo-cond}) is a difficult question. As mentioned in Introduction, there exist various methods \cite{MeyTwe94,LunTwe96,Bax05} providing rates and constants $(\rho,C)$ in Inequality~(\ref{V-geo-cond}), but they yield a rate $\rho$ which is much far from $\rho_V(P)$ except in specific cases. The first purpose of Theorem~\ref{cor-Vk-to-V} below is to derive  a rate $r_k$ in (\ref{V-geo-cond}) from $\rho_k$ in (\ref{Vk-geo-cond}), which is all the more close to $\rho_V(P)$ that $k$ is large and $\rho_k$ in (\ref{Vk-geo-cond}) is close to the so-called second eigenvalue of the finite-rank operator $\Pc_k$. The second purpose of Theorem~\ref{cor-Vk-to-V} is to provide an explicit constant $c_k\equiv c(r_k)$ associated with rate $r_k$ in (\ref{V-geo-cond}). Then Proposition~\ref{pro-P-hatP-en3fois} can be applied to derive a bound for $\| \widehat\pi_k - \pi \|_{TV}$. 

Let us briefly explain why the passage from the $V$-geometric ergodicity of $\Pc_k$  to that of $P$ is theoretically more difficult than the converse one studied in Section~\ref{sec-V-to-Vk}. Assume that Conditions~(\ref{V-geo-cond}), (\ref{C01}) and (\ref{drift-gene-cond-gene}) hold. Let $r>\hat\alpha$ with $\hat\alpha$ given in (\ref{def-hat-alpha}) and let $\vartheta>0$. 
With $\varepsilon_1(r,\vartheta,P)$ defined in (\ref{def-eps-1}), Property (\ref{th-implic1}) reads as follows  
\begin{equation} \label{implic1}
\|\Pc_k  - P\|_{0,1} \leq \varepsilon_1(r,\vartheta,P) \quad \Longrightarrow\quad       \sigma(\Pc_k)\subset \cV(r,\vartheta,P),\  
      \displaystyle \sup_{z\in \cV(r,\vartheta,P)^c} \|(zI- \Pc_k)^{-1}\|_1 < \infty.
\end{equation}
Next, for any fixed $k\in\N^*$, exchanging the role of $P$ and $\Pc_k$ in (\ref{th-implic1}) gives the following implication: 
\begin{equation} \label{implic2}
\|P  - \Pc_{k}\|_{0,1} \leq \varepsilon_1(r,\vartheta,\Pc_{k}) \ \Longrightarrow       \ \ \sigma(P) \subset \cV(r,\vartheta,\Pc_k),\  
      \displaystyle 
      \sup_{z\in \cV(r,\vartheta,\widehat P_k)^c} 
      \|(zI- P)^{-1}\|_1  < \infty.
\end{equation}
There is a significant difference between the implications (\ref{implic1}) and (\ref{implic2}) since the inequality $\|\Pc_k  - P\|_{0,1} \leq \varepsilon_1(r,\vartheta,P)$ in (\ref{implic1}) is satisfied for $k$ large enough from Condition~(\ref{C01}), while the inequality $\|P - \Pc_k\|_{0,1}\leq \varepsilon_1(r,\vartheta,\Pc_{k})$ in (\ref{implic2}) could fail for every $k$. Fortunately, such a failure  cannot occur when Conditions~(\ref{V-geo-cond}), (\ref{C01}) and (\ref{drift-gene-cond-gene}) hold (see Proposition~\ref{th-Vk-to-V}). 

Now we introduce the material used in Theorem~\ref{cor-Vk-to-V}. 
Under Condition~(\ref{drift-gene-cond-gene}), we define the following constants for any $\vartheta>0$, $k\in\N^*$ and $r_k\in(0,1)$, 
with $B$ defined in (\ref{doeb-fortet-dual-itere}): 
\begin{subequations}
\begin{gather} 
\varepsilon_1(k)\equiv \varepsilon_1(r_k,\vartheta,\Pc_k) := \frac{{r_k}^{n_1(k)+n_2(k)}}{8B\big(H_k B +(1-r_k)^{-1}\big)} \label{def-eps-k}
\\[1mm]
\text{with } \quad H_k \equiv H(r_k,\vartheta,\Pc_k) := \sup_{z\in{\cal V}(r_k,\vartheta,\Pc_k)^c} \big\|(zI-\Pc_k)^{-1}\big\|_1  
\label{H-r-vartheta-stat} \\[1mm]
n_1(k) 
:= \left\lfloor \frac{\ln2}{\ln(r_k/\hat\alpha)}\right\rfloor + 1, \quad n_2(k) \equiv n_2(r_k,\vartheta,P_k) := \bigg\lfloor \frac{\ln\big(8B(B+3)\,  {r_k}^{-n_1(k)}H_k\big)}{\ln(r_k/\hat\alpha)}\bigg\rfloor + 1. \label{def-n1-n2-k}
\end{gather}
\end{subequations}

To apply Theorem~\ref{cor-Vk-to-V}, some preliminary (possibly poor) rate $\rho$ in (\ref{V-geo-cond}) is assumed to be available. But it is worth noticing that no constant associated with $\rho$ in (\ref{V-geo-cond}) is required. Inequality~(\ref{unif-bound-rate-bis}) then provides a new rate $r_k$ in (\ref{V-geo-cond}) with an explicit associated constant $c_k$. 
\begin{atheo} \label{cor-Vk-to-V}
Assume that $P$ is $V$-geometrically ergodic with some rate $\rho\in(0,1)$ and that Conditions~ \emph{(\ref{C01})} and \emph{(\ref{drift-gene-cond-gene})} hold. Let $\vartheta$ be such that 
\begin{equation} \label{vartheta-th2}
0<\vartheta < \min\left(\frac{1-\hat\alpha}{3},\frac{1-\rho}{3}\right) 
\end{equation} 
and let ${\cal I}_\vartheta$ denote the following subset of integers: 
\begin{equation} \label{Itheta}
{\cal I}_\vartheta := \big\{\ k\in\N^* : \Pc_k \text{ satisfies \emph{(\ref{Vk-geo-cond})} with some rate } \rho_k < 1-2\vartheta\ \big\}.
\end{equation} 
Then, for any $k\in{\cal I}_\vartheta$ and for any $r_k$ such that 
\begin{equation} \label{cor-cond-r-var-k}
\max(\hat\alpha,\rho_k) + \vartheta < r_k < 1-\vartheta,
\end{equation}
the estimates in (a) and (b) below are valid provided that the convergence condition \emph{(\ref{C01})} holds at rate $\varepsilon_1(k)$, that is 
\begin{equation} \label{ineg-p-pk-inv}
 \|P - \Pc_k\|_{0,1} \leq \varepsilon_1(k) \tag{$\cE_{0,1}(k)$}. 
\end{equation}
\begin{enumerate}[(a)]
	\item The iterates of $P$ converge to $\pi(\cdot)1_{\X}$ with the following explicit rate of convergence: 
	\begin{equation} \label{unif-bound-rate-bis} 
\forall n\geq 0, \ \|P^n - \pi(\cdot)1_\X\|_1 \leq  c_k\, {r_k}^{n+1} 
\text{ with } 
c_k\equiv c_k(r_k,\vartheta) := 
\frac{4(B+1)}{{r_k}^{n_1(k)}(1-r_k)} + \frac{1}{2\, \varepsilon_1(k)}. 
   \end{equation}
   \item We have, with $A$ defined in \emph{(\ref{As-bounded})} and $\Delta_n := \|\Pc_n  - P\|_{0,1}$,
\begin{equation} \label{th-TV-ineg-bis} 
\forall n\geq 1,\quad \| \widehat\pi_n - \pi \|_{TV}\  \leq\  |1-\widehat\pi_n(1_\X)| + 
\frac{L}{1-\delta}\bigg(2c_k  + \frac{A}{\ln {r_k}^{-1}}\, |\ln \Delta_n|\bigg)\Delta_n. 
\end{equation}
\end{enumerate}
\end{atheo}
Assertion~$(a)$ of Theorem~\ref{cor-Vk-to-V} can be established as in  Theorem~\ref{pro-bv} by exchanging the role of $P$ and $\Pc_k$. Assertion~$(b)$ follows from (\ref{unif-bound-rate-bis}) and Inequality (\ref{rem-spec-n-k-ter}) of Proposition~\ref{pro-P-hatP-en3fois}.

Under the assumptions of Theorem~\ref{cor-Vk-to-V}, the conclusions (\ref{unif-bound-rate-bis}) and (\ref{th-TV-ineg-bis}) are valid for $k$ large enough, so that Issue~(Q\ref{2}) is solved. Indeed,  first Proposition~\ref{cor-bv-app} ensures that a sufficient condition for $k$ to belong to $I_{\vartheta}$ is that 
$\|\Pc_k  - P\|_{0,1} \leq \varepsilon_0$ with $\varepsilon_0 \equiv \varepsilon_0(r,\vartheta,P)$ defined in (\ref{def-eps-0-app}) (use that $\max(\hat\alpha,\rho)+\vartheta < 1-2\vartheta$ from (\ref{vartheta-th2})).  Second Proposition~\ref{th-Vk-to-V} ensures that a sufficient condition for the convergence property (\ref{ineg-p-pk-inv}) to hold is that $k\in[\tilde k,+\infty)\cap{\cal I}_\vartheta$ for some integer $\tilde k$ (greater than $k_0$). 
Since integer $k\in{\cal I}_\vartheta$ satisfying (\ref{ineg-p-pk-inv}) may exist for $k <\max(\tilde k,k_0)$, 
an effective application of Theorem~\ref{cor-Vk-to-V} needs to find such a $k$ with a value as small as possible. This can be done in testing the validity of both properties $k\in{\cal I}_\vartheta$ and (\ref{ineg-p-pk-inv}) for increasing values of $k$. Finally note that the bound $H_k$ in (\ref{H-r-vartheta-stat}) can be bounded under Condition~(\ref{Vk-geo-cond}) following the lines of Remark~\ref{sub-rate-H} with  $\widehat P_k$ in place of $P$ (see Lemma~\ref{lem-cont-s} for discrete Markov kernels). An algorithm based on Theorem~\ref{cor-Vk-to-V} is proposed in Subsection~\ref{subsec-ex}  for truncating discrete Markov kernels. Numerical results for random walks on $\X:=\N$ are reported in Subsection~\ref{subsec-num-ex}.

The whole results of \cite{Liv01}, applied with $\Pc_k$ and $P$ viewed as unperturbed and perturbed operators respectively, directly provide Property~(\ref{unif-bound-rate-bis}) when  Condition~(\ref{ineg-p-pk-inv}) is replaced with $\|\Pc_k  - P\|_{0,1} \leq \varepsilon_0(k)$,  where $\varepsilon_0(k) \equiv \varepsilon_0(r_k,\vartheta,\Pc_k)$ is defined as in (\ref{def-eps-0-app}) with $r_k,H_k$ in place of $r,H$. In this case no preliminary bound $\rho$ in (\ref{V-geo-cond}) is required  in Theorem~\ref{cor-Vk-to-V}, and Condition (\ref{vartheta-th2}) is : $0 < \vartheta < (1-\hat\alpha)/3$ (in fact the $V$-geometric ergodicity assumption in Theorem~\ref{cor-Vk-to-V} may be replaced by the quasi-compactness of $P$ on $\cB_1$). When $\varepsilon_0(k)$ is  strictly less than $\varepsilon_1(k)$ given in (\ref{def-eps-k}) and a preliminary bound $\rho$ in (\ref{V-geo-cond}) is known, the alternative Assertion~$(a)$ in Theorem~\ref{cor-Vk-to-V} is of interest since it involves smaller integers $k$. Note that the second eigenvalue of $\Pc_k$ will be all the more tractable that 
$k$ is small since the rank of $\Pc_k$ is expected to increase with respect to $k$. 

When $r_k$ gets closer to $\rho_V(P)$, the constant $c_k$ associated with $r_k$ in (\ref{unif-bound-rate-bis}) goes to $+\infty$. Numerical evidence of this fact is provided by Table~\ref{Ex-rk-ck}. More generally, an efficient use of Theorem~\ref{cor-Vk-to-V} needs to make a trade-off  between the rate $r_k$ and the constant $c_k$ in (\ref{unif-bound-rate-bis}) since the choice of $(\vartheta,r_k)$ greatly affects the bound $H_k$ in (\ref{H-r-vartheta-stat}), thus the constant $c_k$.  

\begin{arem} \label{effectiv-var-total}
 Property~\emph{(\ref{th-TV-ineg-bis})} provides a bound $\|\widehat\pi_n - \pi \|_{TV} = O(\Delta_n|\ln \Delta_n|)$ with computable constants since $r_k$ and $c_k$ are deduced from computations involving the finite-rank operator $\Pc_k$ for some $k$. Consequently, given some error term $\varepsilon>0$, this bound can be used to find an integer $n\equiv n(\varepsilon)\geq1$ such that $\|\widehat\pi_n - \pi \|_{TV}\leq \varepsilon$. However, as illustrated in Section~\ref{sec-appli}, applying Inequality~\emph{(\ref{rem-spec-n-k-bis})} for larger and larger $k$ to test the previous property is more efficient  since the constant $C_k$ in \emph{(\ref{Vk-geo-cond})} is much smaller than the constant $c_k$ in \emph{(\ref{unif-bound-rate-bis})}, which in practice is derived from $C_k$ (see the above comments  on $H_k$). 
\end{arem}

%===========================================================
%===========================================================
%===========================================================
\section{Bounds on the essential spectral radius of quasi-compact kernels on $\cB_1$} \label{sec-mino}
%===========================================================
%===========================================================
%===========================================================

From the definition of $\hat\alpha$ in (\ref{def-hat-alpha}), the new rates in (\ref{Vk-geo-cond}) provided  by Theorem~\ref{pro-bv}, or in (\ref{V-geo-cond}) by Theorem~\ref{cor-Vk-to-V}, are all the more interesting that an accurate bound  of $r_{ess}(P)$ is known (e.g.~see Condition (\ref{cor-cond-r-var-k}) for the new rate in (\ref{V-geo-cond})). 
Estimates of $r_{ess}(P)$ are provided under two different kind of assumptions on the Markov kernel $P$. The first bound for $r_{ess}(P)$ (Theorem~\ref{main}) is derived from a result on positive operators \cite{Sch71}. The second one (Proposition~\ref{pro-qc-bis}) is obtained by duality from the quasi-compactness criterion of \cite{Hen93}. 

First we recall some basic facts on quasi-compactness and on the essential spectral radius of a linear bounded operator $T$ with positive spectral radius on a complex Banach space $(\cB,\|\cdot\|)$. Its spectral radius $r(T):=\lim_n\|T^n\|^{1/n}$, where $\|\cdot\|$ also stands for the operator norm on $\cB$, is assumed to be $1$ (if not, replace $T$ with $r(T)^{-1}T$). We denote by $I$ the identity operator on $\cB$.  The next definitions of quasi-compactness and essential spectral radius are from \cite{Hen93} (to be compared with the reduction of matrices or compact operators). Additional materials are developed in \cite{Nev64,Kre85,Nag86}. 
\begin{adefi} \label{def-q-c} 
$T$ is quasi-compact on $\cB$ if there exist $r_0\in(0,1)$, $m\in\N^*$ and  $(\lambda_i,p_i)\in\C\times \N^*$ for $ i=1,\ldots,m$ such that: 
\begin{subequations}
\begin{equation} \label{noyit}
\cB = \overset{m}{\underset{i=1}{\oplus}} \ker(T - \lambda_i I)^{p_i}\, \oplus H, 
\end{equation}
where the $\lambda_i$'s are such that 
\begin{equation} \label{noyit-lambda}
|\lambda_i| \geq r_0\quad \text{ and } \quad 1\le \dim\ker(T-\lambda_i I)^{p_i} < \infty,
\end{equation}
and $H$ is a closed $T$-invariant subspace such that 
\begin{equation} \label{noyit-H}
\inf_{n\geq1}\big(\sup_{h\in H,\, \|h\|\leq1}\|T^nh\|\big)^{1/n} < r_0.
\end{equation}
\end{subequations}
Then the essential spectral radius of $T$, denoted by $r_{ess}(T)$, is given by 
\begin{equation} \label{def_ress}
r_{ess}(T) = \inf\big\{\ r_0\in(0,1) \text{ such that~we have (\ref{noyit}), (\ref{noyit-lambda}),  (\ref{noyit-H})} \ \big\}.
\end{equation}
\end{adefi}

Note that the infimum bound in the left-hand side of (\ref{noyit-H}) is nothing else but the spectral radius of the restriction of $T$ to $H$. It is well-known that the essential spectral radius of $T$  is usually defined by $r_{ess}(T) := \lim_n(\inf \|T^n-K\|)^{1/n}$, where the infimum is taken over the ideal of compact operators $K$ on  $\cB$. Consequently, $T$ is quasi-compact if and only if there exist some $n_0\in\N^*$ and some compact operator $K_0$ on $\cB$ such that 
$r(T^{n_0}-K_0)<1$. Under the previous condition we have 
\begin{equation} \label{ray-ess-bis}
r_{ess}(T) \leq (r(T^{n_0}-K_0))^{1/n_0}.
\end{equation}
Finally recall that $r_{ess}(T) = r_{ess}(T^p)^{1/p}$ for every $p\geq1$ since $\lim_n(\inf \|T^n-K\|)^{1/n} = \lim_k(\inf \|T^{p k} - K\|)^{1/(p k)}$. 
%
%==========================================================================
%==========================================================================
\subsection{Bounds on $r_{ess}(P)$ under drift/minorization conditions} \label{subsec-qc-drift}
%==========================================================================
%==========================================================================
%

In the next theorem, a simple bound of $r_{ess}(P)$ is given in terms of the parameters of the 
drift/minorization conditions (\ref{inequality-drift})-(\ref{th-small}) below (note that no irreducibility/aperiodicity condition is assumed).  
\begin{atheo} \label{main}
Assume that $P$ satisfies the following drift/minorization conditions: there exist a bounded set $S\in\cX$ and a positive measure $\nu$ on $(\X,\cX)$ such that   
\begin{subequations}
\begin{gather} 
\exists \delta\in(0,1),\ \exists L>0,\quad PV \leq \delta \, V + L\, 1_S, \label{inequality-drift} \\
\forall x\in \X,\ \forall A\in\cX,\ \ \ P(x,A) \geq \nu(1_A)\, 1_S(x).\label{th-small}
\end{gather}
\end{subequations}
Then $P$ is a power-bounded quasi-compact operator on $\cB_1$ with 
\begin{equation} \label{first-r-ess} 
 r_{ess}(P) \leq \frac{\delta\, \nu(1_\X) + \tau}{\nu(1_\X)+\tau} \quad \text{where } \tau:=\max(0,L-\nu(V)).
\end{equation}
\end{atheo}
In the long version \cite{Hen06} of \cite{Hen07}, the  quasi-compactness of $P$ on $\cB_1$ is proved under the assumptions of Theorem~\ref{main}. The bound obtained for $r_{ess}(P)$ in \cite[Th.~IV.2]{Hen06} 
is less tractable than (\ref{first-r-ess}) since it is expressed in terms of the hitting time for $S$. Also mention that, in the unpublished paper \cite{Del97}, the quasi-compactness of Markov kernels is obtained on the subspace of continuous functions of $\cB_1$ under some drift/minorization conditions. No bound on the essential spectral radius is presented in \cite{Del97}.  

The short proof of Theorem~\ref{main} illuminates the role of the drift and minorization conditions to obtain good spectral properties of $P$ on $\cB_1$. In particular, using \cite[Cor.~IV.3]{Hen06}, this provides a simple proof of the fact that under the assumptions of Theorem~\ref{main}, together with irreducibility and aperiodicity assumptions, $P$ is $V$-geometrically  ergodic. This fact is well-known (e.g. see \cite{MeyTwe93}).

\noindent\begin{proof}{}
Condition~(\ref{inequality-drift}) implies that $PV \leq \delta\, V + L\,1_\X$, thus 
\begin{equation} \label{itere-borne}
\sup_{n\geq0}\|P^n\|_1 \leq \frac{1 - \delta + L}{1-\delta}  
\end{equation}
so that $P$ is power-bounded on $\cB_1$. Then, from $P1_\X=1_\X$ and $1_\X\in\cB_1$, we have $r(P)=1$. Moreover, since $\|PV\|_1<\infty$, we deduce from (\ref{th-small}) that $\nu(V)<\infty$. Thus we can define the following rank-one operator on $\cB_1$: $Tf := \nu(f)\, 1_S$. Set $R := P-T$. From $T\geq0$ and from (\ref{th-small}), it follows that $0\leq R \leq P$, so $r(R) \leq 1$. Set $r:=r(R)$. If $r=0$, then $P$ is quasi-compact with $r_{ess}(P) =0$ from (\ref{ray-ess-bis}). Now assume that $r\in(0,1]$. Then, we know from \cite[Appendix, Cor.2.6]{Sch71} that there exists a nontrivial non-negative $\eta\in \cB_1'$  such that $\eta \circ R = r\, \eta$. From $P = T + R$, we have $\eta\circ P = \eta\circ T + r\, \eta$, thus $\eta(P 1_\X) = \eta(1_\X) = \eta(T 1_\X) + r\, \eta(1_\X)$. Hence $\eta(T 1_\X) = (1-r)\eta(1_\X)$, from which we deduce that $$\eta(1_S) = \frac{(1-r)\eta(1_\X)}{\nu(1_\X)} \leq \frac{(1-r)\eta(V)}{\nu(1_\X)}.$$
Next, we have $RV = PV - TV = PV - \nu(V) 1_S \leq  \delta\, V + (L-\nu(V))\, 1_S$. Hence, setting $\tau := \max(0,L-\nu(V))\ge 0$, we obtain 
\begin{equation} \label{r-ineg} 
r\, \eta(V) = \eta(RV) \leq \delta\, \eta(V) + \tau\, \eta(1_S) \leq \delta\, \eta(V) + \tau\frac{(1-r)\eta(V)}{\nu(1_\X)}.
\end{equation}
Since $\eta\neq 0$, we have $\eta(V) > 0$, and since $\delta\in(0,1)$, we cannot have $r=1$. Thus $r\in(0,1)$, and $P$ is quasi-compact from (\ref{ray-ess-bis}) with $r_{ess}(P) \leq r(P-T) = r$. Then (\ref{r-ineg}) gives (\ref{first-r-ess}). 
\end{proof}
\begin{arem} \label{rk-atom}
Recall that $A\in\cX$ is said to be an atom for  $P$ if $P(a,\cdot) =  P(a',\cdot)$ for any $(a,a')\in A^2$.
Any Markov model having  a regenerative structure is concerned with such a property (e.g. see \cite{Num84,Asm03}). 
If $P$ satisfies \emph{(\ref{inequality-drift})} with $S:=A$, then $r_{ess}(P)\leq \delta$.
Indeed, note that $P$ satisfies the minorization condition \emph{(\ref{th-small})}   with $A$ and $\nu(1_{\cdot}) := P(a_0,\cdot)$ for any $a_0\in A$. Choose $L:=\sup_{x\in A}(PV)(x)$ in  \emph{(\ref{inequality-drift})}. Since $A$ is an atom, we have $L=(PV)(a_0) = \nu(V)$ so that $\tau=0$ in \emph{(\ref{first-r-ess})}. 
\end{arem}
%
%
%=====================================================================
%=====================================================================
\subsection{Bound on $r_{ess}(P)$ under a weak drift condition} \label{sub-suf-K}
%=====================================================================
%=====================================================================

The key idea to obtain quasi-compactness in Proposition~\ref{pro-qc-bis} below  is to use the dual Doeblin-Fortet inequality  obtained in Lemma~\ref{lem-D-F}. Despite its great simplicity, this duality approach seems to be unknown in the literature. 
In particular it allows us to greatly simplify the arguments used in \cite{Wu04} since the well known statement \cite[Cor.~1]{Hen93} gives the bound $r_{ess}(P)\leq \delta$ provided that $P^\ell$ is compact from $\cB_0$ to $\cB_1$ for some $\ell\geq1$.  This compactness condition is much simpler than the assumptions of \cite{Wu04} based on sophisticated parameters $\beta_w(P)$ and $\beta_\tau(P)$ as measure of non-compactness of $P$.  Precise comparisons with \cite{Wu04} and complements  are  presented in \cite[Sect.~2.3]{GuiHerLed11}.  
Simple sufficient conditions for this compactness property are presented in \cite{GuiHerLed11}. For instance, this holds for any discrete Markov chains or for functional autoregressive models on $\X:=\R^q$ with absolutely continuous noise with respect to the Lebesgue measure. Finally note that Equality $r_{ess}(P) = \delta$ holds in several cases (see \cite{GuiHerLed11}).
\begin{apro} \label{pro-qc-bis} 
Assume that $P$ satisfies the following condition 
\begin{equation} \label{drift-gene-cond-gene-P} 
\exists\,  \delta\in(0,1),\ \exists L>0, \quad P V \leq \delta V + L\, 1_{\X}  \tag{\text{WD}}
\end{equation}
and that $P^\ell : \cB_0\r\cB_1$ is compact for some $\ell\geq1$. Then $P$ is a power-bounded quasi-compact operator on $\cB_1$, with $r_{ess}(P) \leq \delta$. 
\end{apro}
\begin{proof}{} 
Iterating (\ref{drift-gene-cond-gene-P}) ensures that $P$ is power-bounded on $\cB_1$ (see (\ref{itere-borne})). Since we have 
$r_{ess}(P) = (r_{ess}(P^\ell))^{1/\ell}$, 
we only consider the case $\ell:=1$, that is $P : \cB_0\r\cB_1$ is compact. 
Let $\cB_1'$ and $\cB_0'$ denote the dual spaces of $\cB_1$ and $\cB_0$ respectively. Let $P'$ denote the adjoint operator of $P$ on $\cB_1'$. In fact, we prove that $P'$ is a quasi-compact operator on $\cB_1'$ with $r_{ess}(P') \leq \delta$, so that $P$ satisfies the same properties on $\cB_1$. Since the operator $P : \cB_0\r\cB_1$ is assumed to be compact, so is $P' : \cB_1'\r\cB_0'$. Then we deduce from the Doeblin-Fortet inequality of Lemma~\ref{lem-D-F} and from \cite[Cor.~1]{Hen93} that $P'$ is a quasi-compact operator on $\cB_1'$ with $r_{ess}(P') \leq \delta$.  
\end{proof}
\begin{arem} \label{rk-set-itere}
If Conditions~\emph{(\ref{inequality-drift})-(\ref{th-small})} are fulfilled for some iterate $P^N$ in place of $P$ (with parameters $\delta_N<1$, $L_N>0$ and positive measure $\nu_N(\cdot)$), then the conclusions of Theorem~\ref {main} are valid with \emph{(\ref{first-r-ess})} replaced by 
$$r_{ess}(P) = r_{ess}(P^N)^{1/N} \leq \left(\frac{\delta_N \nu_N(1_\X) + \tau_N}{\nu_N(1_\X)+\tau_N}\right)^{1/N}
\quad \text{where } \tau_N:=\max(0,L_N-\nu_N(V)).$$ 
Similarly Proposition~\ref{pro-qc-bis} still holds when \emph{(\ref{drift-gene-cond-gene-P})} is replaced by the following condition 
\begin{equation*} 
\exists \delta\in(0,1), \ \exists N\in\N^*,\ \exists L >0,\quad P^NV \leq \delta^N\, V + L\, 1_{\X}.  
\end{equation*} 
Indeed the dual Doeblin-Fortet inequality of  Lemma~\ref{lem-D-F} extends by replacing ${P'},\delta$ by ${P'}^N$ and $\delta^N$ respectively. 
\end{arem}

%============================================================================ 
%============================================================================
\section{Application to truncation of discrete  Markov kernels} \label{sec-appli}
%============================================================================
%============================================================================
In this section we assume that $P:=(P(i,j))_{(i,j)\in\N^2}$ is a Markov kernel on $\X:=\N$. Let $B_k:=\{0,\ldots;k\}$ for any $k\ge 1$. We consider the $k$-th truncated (and augmented) matrix $P_k$: 	
\[ \forall (i,j)\in {B_k}^2,\quad P_k(i,j) := \begin{cases}
P(i,j) & \text{ if $(i,j)\in B_k\times B_{k-1}$} \\
\sum_{\ell\ge k}P(i,\ell) & \text{ if $(i,j)\in B_k\times \{k\}$}.
\end{cases}
\]
Such a matrix is generally called a linear augmentation (in the last column here) of the $(k+1)\times (k+1)$ northwest corner truncation of $P$. Other kinds of augmentation, as the 
censored Markov chain \cite{ZhaLiu96}, could be considered. 
Truncation approximation of an infinite stochastic matrix has a long story (e.g. see \cite{Sen81,Twe98,Liu10} and the references therein). 

The associated (extended) sub-Markov kernel $\Pc_k$  on $\N$ is defined by: 
\begin{equation*} 
\forall (i,j)\in\N^2,\quad \Pc_k(i,j) := \begin{cases}
P_k(i,j) & \text{ if $(i,j)\in {B_k}^2$} \\
\quad 0 & \text{ if $(i,j)\notin {B_k}^2$}.
\end{cases} 
\end{equation*}
%===================
\subsection{Theorem~\ref{cor-Vk-to-V} for truncation of discrete Markov kernels} \label{subsec-appli}
%====================
Consider any unbounded increasing sequence $V:=\{V(j)\}_{j\in\N}\in[1,+\infty)^\N$ with $V(0)=1$, and the associated  weighted space $(\cB_1,\|\cdot\|_1)$ given by 
$$\cB_1 := \big\{ f\in\C^\N\, : \|f\|_1 = \sup_{j\in\N}|f(j)|{V(j)}^{-1} < \infty\big\}.$$
The following lemma helps us to check the assumptions of Theorems~\ref{pro-bv} and \ref{cor-Vk-to-V}.
\begin{alem} \label{lem-gene-trunc} 
Assume that $P$ satisfies Condition~\emph{(\ref{drift-gene-cond-gene-P})} with parameters $(\delta,L)$. Then the following assertions hold. 
\begin{enumerate}[(i)]
	\item $P$ is a power-bounded quasi-compact operator on $\cB_1$ with $r_{ess}(P) \leq \delta$. 
	\item Condition~\emph{(\ref{drift-gene-cond-gene})} is fulfilled, that is:
$\forall k\in\N^*\cup \{\infty\},\ \Pc_kV \leq \delta V + L\, 1_{\N}$. 
  \item Property~\emph{(\ref{C01})}, that is $\lim_k \|\Pc_k - P\|_{0,1} = 0$, holds true since 
\begin{equation} \label{cont-trunc}
\forall k\in\N^*,\quad  \|\Pc_k  - P\|_{0,1} \leq \frac{K}{V(k)} \qquad \text{with $K:=\max\big(2(\delta+L),1\big)$}.
\end{equation}
\end{enumerate}
\end{alem}
If the stochastic matrix $P_k$ has an invariant probability measure $\pi_k$, then $\widehat\pi_k$ is the  probability measure on $\N$ defined by 
$$\forall j\in\N,\quad \widehat\pi_k(\{j\}) := \begin{cases}
\pi_k(\{j\}) & \text{ if $j\in B_k$} \\
\quad 0 & \text{ if $j\notin B_k$.}
\end{cases} 
$$
\begin{arem} \label{Annexe_Continuity}
Assume that $\limsup_{i} (PV)(i)/V(i) >0$ and that $P$ satisfies \emph{(\ref{drift-gene-cond-gene-P})}, so that the infimum of $\delta$ such that \emph{(\ref{drift-gene-cond-gene-P})} holds is non zero (these assumptions are satisfied in almost all $V$-geometrically ergodic models). Then the strong convergence property $\lim_k\|P - \Pc_k\|_{1}=0$ required in the standard perturbation theory does not hold. Indeed, using $\Pc_k V(i)=0$ when $i\notin B_k$,  we obtain  
$$\sup_{i\notin B_k}  \frac{(PV)(i)}{V(i)} = \sup_{i \notin B_k}  \frac{|(PV)(i)-(\Pc_kV)(i)|}{V(i)} \leq  \| PV - \Pc_k V\|_1  \le  \|P - \widehat P_k\|_{1}.$$
If $\lim_k\|P - \Pc_k\|_{1}=0$, then $\lim_k\sup_{i\notin B_k}  (PV)(i)/V(i)=0$ which cannot hold from the hypothesis. 
Actually using this strong convergence condition leads to difficulties or restrictions in other approximation questions. For instance in \cite{KonMey05}, some iterate $P^N$ of the Markov kernel $P$ is approached  by  finite rank kernels in operator norm on $\cB_1$ (for other purpose than truncation issue). Thus $P^N$ is compact on $\cB_1$.  Consequently  this property in operator norm $\|\cdot\|_1$ leads to suppose that $r_{ess}(P)=0$ which is restrictive for $V$-geometrically ergodic kernel $P$ since $r_{ess}(P)<1$ but $r_{ess}(P)\neq 0$ in general. 
\end{arem}
\begin{proof}{ of Lemma~\ref{lem-gene-trunc}}
Assertion $(i)$ follows from Proposition~\ref{pro-qc-bis} since the injection from $\cB_0$ into $\cB_1$ is compact from Cantor's diagonal argument and $\lim_j V(j) = +\infty$. Next let $i \in B_k$. Then we obtain using $V(j) \geq V(k)$ for $j\geq k$ 
\begin{eqnarray*}
(\Pc_kV)(i)  & = & \sum_{j=0}^{k-1} P(i,j)V(j) +  V(k) \sum_{j\geq k}  P(i,j) \, \leq\,  PV(i) 
\end{eqnarray*}
so $(\Pc_kV)(i) \leq \delta V(i) + L$ from (\ref{drift-gene-cond-gene-P}). If $i\notin B_k$, then $(\Pc_kV)(i) = 0$. This proves $(ii)$. To derive~$(iii)$, consider any  $f\in\cB_{0}$ such that $\|f\|_0\leq1$. First, we obtain for every $i \in B_k$:
\begin{eqnarray*}
\big|(Pf)(i) - (\Pc_kf)(i)\big| &=& \big|  \sum_{j\geq k}  P(i,j)(f(j) - f(k))  \big| \leq 2\, \sum_{j\geq k}  P(i,j)   = 2\, \sum_{j\geq k}  P(i,j) V(j)\frac{1}{V(j)} \\ 
&\leq& \frac{2(PV)(i) }{V(k)} \leq \frac{2(\delta+L)}{V(k)}\, V(i) \quad \text{(since $PV\leq (\delta+L)\, V$ from (\ref{drift-gene-cond-gene-P}))}.
\end{eqnarray*}
Second let $i\notin B_k$. Then $(\Pc_kf)(i)=0$, so that  
$|(Pf)(i) - (\Pc_kf)(i)| \le \sum_{j\in \N} P(i,j)|f(j)|\le 1$. 
This gives (\ref{cont-trunc}) in $(iii)$. 
\end{proof}

To conclude this subsection we present two lemmas. The first one provides a useful bound for the constants $C_k$ in (\ref{Vk-geo-cond}) and  $H_k$  in (\ref{H-r-vartheta-stat}). 
It is convenient to consider $P_k$ as an endomorphism on the finite dimensional space of functions $h : B_k\r\C$ equipped with the following norm (still denoted by $\|\cdot\|_1$ for the sake of simplicity) defined by: $\|h\|_1 := \sup_{i\in B_k}|h(i)|/V(i)$. The associated operator norm is also denoted by $\|\cdot\|_1$. 

\begin{alem} \label{lem-cont-s}
The matrix $\Pc_k$ is assumed to be $V$-geometrically ergodic, that is $1$ is a simple eigenvalue of the stochastic matrix $P_k$ and is the unique eigenvalue of modulus one. Suppose that an explicit upper bound  $\widetilde\rho_k\in(0,1)$ of the second eigenvalue of $P_k$ is known. Let any $\rho_k$ be such that 
\begin{equation} \label{rho-k-proc}
\widetilde\rho_k < \rho_k <1 
\end{equation}
and define $s\equiv s(\rho_k)\in\N^*$ as the smallest integer such that 
\begin{equation} \label{cont-itere}
\|{P_k}^{s} - \pi_k(\cdot)1_{B_k}\|_1 \leq {\rho_k}^{s}.
\end{equation}
Then, we obtain the following estimate 
\begin{eqnarray}
 & \forall n\geq0, & \|{\Pc_k}^{\, n} - \widehat\pi_k(\cdot)1_{B_k}\|_1 
\leq C_k\, {\rho_k}^{n} \leq \overline{C}_k\, {\rho_k}^{ n} \nonumber \\[1mm]
&& \text{with }\ C_k := \frac{\max_{0\leq r \leq s-1}\|{P_k}^{r} - \pi_k(\cdot)1_{B_k}\|_1}{{\rho_k}^{s-1}} \quad \overline{C}_k := \frac{1-\delta+2L}{(1-\delta)\rho_k^{\, s-1}}. \label{Ck-proc}
\end{eqnarray}
Moreover, for any $(r_k,\vartheta)\in(0,1)^2$ such that $\rho_k+\vartheta < r_k < 1-\vartheta$,  the following bounds hold for $H_k$ defined in \emph{(\ref{H-r-vartheta-stat})}:  
\begin{equation} \label{Hk-proc1}
H_k 
\ \leq \ \overline{H}_k := \max\left(\frac{L+ C_k(1-\delta)}{\vartheta(1-\delta)}\, ,\, \frac{1}{\vartheta}\right)  \ \leq \ {\overline{H}_k}' := \max\left(\frac{L+ \overline{C}_k(1-\delta)}{\vartheta(1-\delta)}\, ,\, \frac{1}{\vartheta}\right)  . 
\end{equation}
\end{alem}
\begin{proof}{}
Let $n\geq0$. Writing $n=sm+r$ with $r\in\{0,\ldots,s-1\}$ we deduce from (\ref{cont-itere}) that
\begin{eqnarray*}
\|{P_k}^{n} - \pi_k(\cdot)1_{B_k}\|_1 &=& \big\|\big({P_k}^{r} - \pi_k(\cdot)1_{B_k}\big)\circ\,  \big({P_k}^{sm} - \pi_k(\cdot)1_{B_k}\big)\big\|_1  \\
&\leq&  \max_{0\leq r \leq s-1}\|{P_k}^{r} - \pi_k(\cdot)1_{B_k}\|_1\, {\big\|{P_k}^{s} - \pi_k(\cdot)1_{B_k}\big\|_1}^m  \\
&\leq& C_k\, {\rho_k}^{n}. 
\end{eqnarray*}
with $C_k$ given in (\ref{Ck-proc}). Inequality $C_k \leq \overline{C}_k$ follows from 
$$\forall r=0,\ldots,s-1,\quad \|{P_k}^{r} - \pi_k(\cdot)1_{B_k}\|_1 \leq \|{P_k}^{r}\|_1 +  \|\pi_k(\cdot)1_{B_k}\|_1 \leq 1+\frac{2L}{1-\delta}$$
since Lemma~\ref{lem-gene-trunc}$(ii)$ gives: first ${P_k}^{r}V_{\mid B_k} \leq \delta^rV_{\mid B_k} + L(1-\delta^r)/(1-\delta)$, thus $\|{P_k}^{r}\|_1\leq 1 + L/(1-\delta)$; second  
$\pi_k(V_{\mid B_k}) \leq L/(1-\delta)$ (see (\ref{maj-pi-V})). The first assertion of Lemma~\ref{lem-cont-s} is proved. 

Now let $(r,\vartheta)\in(0,1)^2$ be such that $\rho_k+\vartheta < r_k < 1-\vartheta$. Since  $r_k > \rho_k$, we obtain 
$${\cal V}\big(r_k,\vartheta,P_k\big)^c = \overline{D}(0,r_k)^c \cap \overline{D}(1,\vartheta)^c.$$
Then Remark~\ref{sub-rate-H} applied to $P_k$ (with rate $\rho_k$ and bound $C_k$) gives 
\begin{equation}  
\sup_{z\in{\cal V}(r_k,\vartheta,P_k)^c} \|(zI-P_k)^{-1}\|_1  =  \sup_{z\in \overline{D}(0,r_k)^c \cap \overline{D}(1,\vartheta)^c} \|(zI-P_k)^{-1}\|_1 \le \frac{\pi_k(V_{\mid B_k})+C_k}{\vartheta}. \label{res-Pk-proc}
\end{equation} 
Moreover observe that we have for any $z\in\C^*\setminus\sigma(P_k)$, 
\begin{equation} \label{res-PK-hat-Pk-proc}
\forall g\in\cB_1,\ \forall i\in\N,\quad \big((zI - {\Pc_k})^{-1} g\big)(i) := \left \{
    \begin{array}{lll}
     \big((zI_k - P_k)^{-1}g_{_{B_k}}\big)(i) &\text{if $i \in B_k$} \\
      g(i)/z & \text{if $i\notin B_k$}.  
    \end{array}
    \right. 
\end{equation}
If $z\in{\cal V}(r_k,\vartheta,\Pc_k)^c$, then $|z| > r_k > \vartheta$, so that  (\ref{res-Pk-proc}) and (\ref{res-PK-hat-Pk-proc}) give 
$$H_k = \sup_{z\in{\cal V}(r_k,\vartheta,\Pc_k)^c} \|(zI-\Pc_k)^{-1}\|_1 \leq \frac{\max\big(\pi_k(V_{B_k})+C_k\, ,\, 1\big)}{\vartheta}.$$
Then the bounds on $H_k$  in (\ref{Hk-proc1}) are deduced from (\ref{maj-pi-V}) and $C_k \leq \overline{C}_k$. 
\end{proof}

The results for truncation of discrete Markov kernels on $\X:=\N$ deduced from Theorem~\ref{cor-Vk-to-V}  are gathered in the following lemma. This is the basic material for the algorithm proposed in the next subsection. 
\begin{alem}\label{lem-TV-final} Let $P$ be a $V$-geometrically ergodic Markov kernel on $\N$ with some rate $\rho\in(0,1)$ and satisfying Condition~\emph{(\ref{drift-gene-cond-gene-P})} with parameters $(\delta,L)$. 
The $k$-th truncated matrix $P_k$ is supposed to satisfy the assumptions of Lemma~\ref{lem-cont-s}. Pick $\rho_k$ as in   \emph{(\ref{rho-k-proc})}. Let $\vartheta$ be such that 
\begin{equation} \label{vartheta-th2-bis}
0<\vartheta < \min\bigg(\frac{1-\delta}{3},\frac{1-\rho}{3},\frac{1-\rho_k}{2}\bigg) .
\end{equation} 
Let $r_k \in(0,1)$ be such that $\max(\delta,\rho_k)+\vartheta < r_k  < 1-\vartheta$. If $k$ is such that $V(k) \geq K/\varepsilon_1(k)$ with $\varepsilon_1(k), K$ defined in \emph{(\ref{def-eps-k})} and \emph{(\ref{cont-trunc})} respectively, then we have the following estimates
\begin{subequations}
\begin{eqnarray} 
\forall n\geq 0, \quad \|P^n - \pi(\cdot)1_\N\|_1  &\leq & c_k\, {r_k }^{n+1} \label{unif-bound-rate-ter} \\
\forall n\geq n_K, \quad \| \widehat\pi_n - \pi \|_{TV}\  & \leq & \frac{L\, K}{1-\delta}\left(2c_k  + \frac{A \, \ln V(n)}{\ln({r_k }^{-1})}\,\right)\frac{1}{V(n)} \label{th-TV-ineg-bis-discret} 
\end{eqnarray} 
where $n_K := \min\{n\in\N^*, V(n) \geq K \}$, $A$ is introduced in \emph{(\ref{As-bounded})} and  
$c_k$ in \emph{(\ref{unif-bound-rate-bis})}. 
\end{subequations}
\end{alem}
To derive the final form (\ref{th-TV-ineg-bis-discret}) of (\ref{th-TV-ineg-bis}), note that, for $n\geq n_K$, we have $1\ge \Delta_n\ge \|P 1_\N - \Pc_n 1_{\N} \|_1 = \| 1_{\N \backslash B_n} \|_1 = 1/V(n)$ so that $|\ln \Delta_n | \le \ln V(n)$.

Recall that, if no rate $\rho$ in (\ref{V-geo-cond}) is known, then Lemma~\ref{lem-TV-final} applies provided that $k\geq k_0$, with $k_0$ defined as in (\ref{con-k0-r-var}) where $\varepsilon_0$ is replaced by $\varepsilon_0(k)$ given as in (\ref{def-eps-0-app}) using $(r_k,H_k)$ in place of $(r,H)$ (see the comments after Theorem 4.1).  In this case, delete $(1-\rho)/3$ in (\ref{vartheta-th2-bis}).

 Property~(\ref{th-TV-ineg-bis-discret}) provides a bound $\|\widehat\pi_n - \pi \|_{TV} = O(\ln(V(n))/V(n))$ with computable constants since $r_k$ and $c_k$ are deduced from calculations involving the matrix $\Pc_k$ for some $k$. Consequently, given some error term $\varepsilon>0$, this bound can be used to find an integer $n\equiv n(\varepsilon)\geq1$ such that $\|\widehat\pi_n - \pi \|_{TV}\leq \varepsilon$ (see Table~\ref{Ex_borne_Pn}-(b) for an illustration). However, as mentioned in Remark~\ref{effectiv-var-total}, using the direct  Inequality~(\ref{rem-spec-n-k-bis}) is the best way to estimate the total variation distance between $\pi$ and $\widehat\pi_k$. This point is illustrated in Subsection~\ref{subsec-num-ex} (see Table~\ref{Ex_borne_pi}). 
To specify~(\ref{rem-spec-n-k-bis}), observe that the first term of the bounds satisfies from $\phi_k:= 1_{B_k}$ and (\ref{maj-pi-V}) 
\begin{equation*} 
\widehat\pi_k(1_\N)\, |1-\pi(\phi_k)|\leq \sum_{j\notin B_k} \pi(j) \leq \frac{1}{V(k)}\sum_{j\notin B_k} \pi(j)V(j) \leq \frac{\pi(V)}{V(k)} \leq \frac{L}{(1-\delta)V(k)}
\end{equation*}
so that Inequality~(\ref{rem-spec-n-k-bis}) reads as follows 
\begin{equation} \label{rem-spec-n-k-bis-discret}
\| \widehat\pi_k - \pi \|_{TV}\  \leq \  \frac{L}{1-\delta}\bigg(1+ \frac{2 K C_k}{\rho_k} + \frac{A K }{\ln({\rho_k}^{\, -1})}\, \ln V(k)\bigg)\frac{1}{V(k)} .
\end{equation}
%

%===================
\subsection{Algorithm for assessing the bounds} \label{subsec-ex} 
%====================

 Let $P$ be a $V$-geometrically ergodic Markov kernel on $\X:=\N$ with some rate $\rho\in(0,1)$ and satisfying Condition~(\ref{drift-gene-cond-gene-P}) with parameters $(\delta,L)$. Assume that, for any $k\ge 1$, $1$ is a simple eigenvalue of the stochastic matrix $P_k$ and is the unique eigenvalue of modulus one.
Lemma~\ref{lem-TV-final} is used to propose the following generic algorithm which allows us to assess the bounds (\ref{unif-bound-rate-bis}) and (\ref{th-TV-ineg-bis}) provided by Theorem~\ref{cor-Vk-to-V} for $\|P^n -\pi(\cdot)1_{\N} \|_1$ and  $\|\pi-\widehat\pi_k\|_{TV}$ respectively. Numerical illustrations are proposed in the next subsection.

First compute the constants in (\ref{As-bounded}), (\ref{th-TV-ineg-bis-discret}), (\ref{def-hat-alpha}) and  (\ref{doeb-fortet-dual-itere}):  
$$A:= 1+ \frac{L}{1-\delta} \quad K := \max\big(2(\delta+L),1\big)  \quad\hat\alpha=\max(\delta,r_{ess}(P))=\delta,\quad B:= \frac{L}{1-\hat\alpha} .$$
\leftmargini=1.5em
\begin{enumerate} \setlength{\itemsep}{0mm}
  \item \label{step1}Pick a (small) $k\in\N^*$. 
	\item \label{step2} Compute the second eigenvalue $\widetilde\rho_k$ of $P_k$. 
	\item \label{step3} Pick $\rho_k \in (\max(\hat\alpha,\widetilde\rho_k),1)$.
	\item \label{step4} Compute the invariant probability measure $\pi_k$ of $P_k$ and the matrix $G_k := P_k - \pi_k(\cdot)1_{B_k}$. 
	\item \label{step5} Compute $s:=\inf\{n\in \N^* : \|{G_k}^{n}\|_1 \leq {\rho_k}^{n}\}$ (see  (\ref{cont-itere})).
  \item \label{step6} Compute  $C_k := \max\{\,\|{G_k}^r\|_1/{\rho_k}^{s-1}: 0\leq r \leq s-1\,\}$.
  	\item \textbf{Compute the bound for $\|\pi-\widehat\pi_k\|_{TV}$ in (\ref{rem-spec-n-k-bis-discret})}.
	\item \label{step7} Pick $\vartheta$ and $r_k$ such that 
$0<\vartheta < \min\left( \frac{1-\rho_k}{2},\frac{1-\rho}{3},\frac{1-\delta}{3}\right)$ and $\rho_k + \vartheta < r_k < 1-\vartheta$.
  \item \label{step8} Compute $ n_1(k) := \left\lfloor \displaystyle \frac{\ln2}{\ln(r_k/\delta)}\right\rfloor + 1,  n_2(k) := \bigg\lfloor \frac{\ln\big(8B(B+3)\,  {r_k}^{-n_1(k)} H_k\big)}{\ln(r_k/\delta)}\bigg\rfloor + 1$, $H_k := \frac{L+ C_k(1-\delta)}{\vartheta(1-\delta)}$ and  $$\varepsilon_1(k):=\frac{{r_k}^{n_1(k)+n_2(k)}}{8B(H_kB+(1-r_k)^{-1})} \qquad k_1\equiv k_1(k):=\min\big\{n: V(n)\ge K/\varepsilon_1(k)\big\}.$$
  \item \label{step9} Check Condition~(\ref{ineg-p-pk-inv}), that is:
\begin{enumerate}
	\item[] if $k < k_1$ then go to Step~\ref{step1} setting $k:=k+1$; 
		
 \textbf{else compute constant $c_k$ in (\ref{unif-bound-rate-bis}) and use bounds in (\ref{unif-bound-rate-ter})-(\ref{th-TV-ineg-bis-discret}) with $\boldsymbol{(c_k,r_k)}$}.
\end{enumerate}
\end{enumerate}
You can use $\overline{C}_k$ in (\ref{Ck-proc}) in place of $C_k$ in Step~\ref{step6}, so that $C_k,H_k$ are replaced by $\overline{C}_k,{\overline{H}_k}'$ in the sequel of the algorithm.
%===================
\subsection{A numerical example} \label{subsec-num-ex}
%====================

Let us consider an instance of random walk on $\X:=\N$ with identically distributed bounded increments, that is a Markov chain with transition kernel $P$ defined, for some $c,g,d\in\N^*$, by  
\begin{gather*}
\forall i\in\{0,\ldots,g-1\},\quad \sum_{j=0}^c P(i,j)=1; \label{Def_Bord_NHRW} \\
\forall i\ge g, \forall j\in\N, \quad P(i,j) = 
\begin{cases}
 a_{j-i} & \text{ if }\  i-g\leq j \leq i+d \\
 0 & \text{ if not}  \\
\end{cases} 
\end{gather*} 
where $(a_{-g},\ldots,a_d)\in[0,1]^{g+d+1}$ satisfies $\sum_{k=-g}^{d} a_k=1$ for all $i\ge g$. This kind of kernels arises, for instance, from time-discretization of Markovian queuing models. Under the negative mean increment condition 
\begin{equation*} 
\sum_{k=-g}^{d} k\, a_{k}\, < 0 
\end{equation*}
the following properties are known:
\begin{enumerate}
\item there exists $\gc>1$ such that $\phi(\gc):=\sum_{k=-g}^{d} a_k \gc^k = \min_{\gamma >1} \phi(\gamma)<1$. Let $V \equiv \{\gc^n\}_{n\in\N}$.
\item $P$ satisfies (\ref{drift-gene-cond-gene-P}) with $\delta = \phi(\gc)$ and $L=\max(0,\sum_{j=0}^{c}P(i,j)\gc^{j}-\phi(\gc):i=1,\ldots,g-1)$. Moreover,  $r_{ess}(P)=\delta$ (see \cite{GuiHerLed12}).
	\item $P$ is $V$-geometrically ergodic with an invariant probability measure $\pi$ such that $\pi(V)<\infty$.
\end{enumerate}
A general procedure to estimate the infimum bound $\rho_V(P)$ of the rates $\rho$  in Inequality~(\ref{V-geo-cond}) for such models is given in \cite{GuiHerLed12}. But the assessment of the constant is not addressed. The next tables give such constants as well as explicit bounds for the total variation distance $\|\pi - \widehat\pi_k\|_{TV}$. To the best of our knowledge, such results are not known for this kind of models. 

In the present context, the integer $k_1$ of Step~\ref{step8} has the following form 
 \begin{equation*} 
 k_1:=\left\lfloor\frac{1}{\ln\gc}\bigg(\ln\big(8BK\big) + \ln\big(H_k B + (1-r_k)^{-1}\big) +  \big(n_1(k)+n_2(k)\big)\ln(r_k^{\, -1})\bigg)\right\rfloor + 1.
\end{equation*}

As a matter of example, we take $c:=2, g:=2, d:=1$ with $a_{-2},a_1>0$ and boundary probabilities   
\begin{gather*}
P(0,0) = a \in(0,1),\quad P(0,1) = 1-a,\quad P(1,0) = b\in(0,1), \quad P(1,2) = 1-b. \label{bound-deux-vois} 
\end{gather*}
It is easily seen that $P$ and $\{P_k\}_{k\ge 1}$ are irreducible and aperiodic. Since $P(2,0)=a_{-2}>0$ and $P(0,2)=0$, $P$ is not reversible. 
The form of boundary probabilities is chosen for convenience. Other (finitely many) boundary probabilities could be considered provided that $P$ and $P_k$ $k\ge 1$ are irreducible and aperiodic. In order to provide numerical evidence for the effectiveness of the algorithm, we also specify the values 
$$a_{-2} := 1/2,\ a_{-1} := 1/3,\ a_0 := 0,\ a_1 := 1/6.$$ 
We obtain $\gc \approx 2.18$ and $\delta \approx 0.621$. Moreover, using the procedure in \cite{GuiHerLed12}, we can obtain a numerical approximation of $\rho_V(P)$, so that any bound on this estimate gives the required initial value of $\rho$ in (\ref{V-geo-cond}) to use Theorem~\ref{cor-Vk-to-V}. 
Note that $P$ is stochastically monotone iff $a\ge b\ge5/6$.  

We report in Table~\ref{Ex_borne_pi} and Table~\ref{Ex_borne_Pn}-(a) the value of the different bounds (\ref{unif-bound-rate-ter})-(\ref{th-TV-ineg-bis-discret}) for $(a,b):=(1/2,1/2)$ which gives a kernel $P$ which has no eigenvalue in the annulus $\{z\in\C : \delta< |z|<1\}$ so that $\rho_V(P)=\delta\approx 0.621$ (see \cite{GuiHerLed12} for details). $P$ is not reversible and not stochastically monotone. Therefore the results of \cite{Twe98,RobTwe99,Bax05} are not relevant since the bound for $\|\widehat\pi_k   -  \pi\|_{TV}$ in \cite{Twe98} is only obtained  for stochastically monotone Markov kernels, and  the rate $\rho$ in (\ref{V-geo-cond}) obtained in \cite{RobTwe99,Bax05} is very close to $1$ (and far from the best rate) in the non reversible case. We also report in Table~\ref{Ex_borne_pi} the results obtained from the direct Inequality~(\ref{rem-spec-n-k-bis-discret}). 
These results are in agreement with the expected fact that Inequality~(\ref{rem-spec-n-k-bis-discret}) gives better bounds for $\|\widehat\pi_k -\pi\|_{TV}$ than (\ref{th-TV-ineg-bis-discret}). We also report in Table~\ref{Ex_borne_Pn}-(b) the value of $n(\varepsilon)$ such that 
 $\|\widehat\pi_n -\pi\|_{TV}\le \varepsilon$ using (\ref{th-TV-ineg-bis-discret}).
\begin{arem}
Note that the assumptions for using the algorithm (see Lemma~\ref{lem-TV-final}) are not sensitive to approximate values of $\gc$ and $\delta$ since properties as $r_{ess}(P)=\delta=\phi(\gamma)<1$ and $V$-geometric ergodicity are valid for any $\gamma \in (1,\gamma_0)$ with $\gamma_0=5.541$.
\end{arem}
\begin{table}[ht]
\centering
\setlength{\extrarowheight}{2pt}
{\small
\begin{tabular}{r||c|c|c|}  \hline 
 \multicolumn{1}{c}{} & \multicolumn{3}{c}{$\boldsymbol{(a,b)=(1/2,1/2)}$}    \\ \hline\hline          
 $\boldsymbol{r_{45}}$  & 0.87 &   0.78        &  0.76   \\\hline
 $\boldsymbol{c_{45}}$ & $1.924\times 10^7$ & $4.610\times 10^{11}$ & $1.348\times 10^{14}$   \\\hline
 \end{tabular}
\caption{Improvement of the rate vs degradation of the constant using Theorem~\ref{cor-Vk-to-V}}
\label{Ex-rk-ck}
}
\end{table}

\begin{table}[ht]
\setlength{\extrarowheight}{2pt}
\centering
{\scriptsize 
\begin{tabular}{>{\bfseries}r||c|c|c||c|c||c|c|c|c|c| }  \hline 
 \multicolumn{1}{c}{}    & \multicolumn{9}{c}{$\boldsymbol{(a,b)=(1/2,1/2)}$} \\ \hline 
 $\boldsymbol{k}$ & $|\widetilde\rho_k|$  & $\rho_k$  & $s$  & $C_k$ & \textbf{(\ref{rem-spec-n-k-bis-discret})}  & $\vartheta$  &$r_k$ & $k_1$ & $c_k$  &  {\bf (\ref{th-TV-ineg-bis-discret}) for  $\boldsymbol{n:=k}$}                         \\\hline\hline
    15 & 0.6018 & 0.75 & 4 & 4.1539 & $\boldsymbol{8.44 \times 10^{-2}}$ & 0.09 & 0.9 & 20 &  &  \\\hline
   25 & 0.6142 & 0.75 & 4 & 4.1540 & $\boldsymbol{5.712 \times 10^{-5}}$ & 0.09 & 0.9 & 20 & $\boldsymbol{4.715 \times 10^{5}}$ & $\boldsymbol{1.112 \times 10^{-1}}$ \\\hline
  35 & 0.6177 & 0.75 & 4 & 4.1540 & $\boldsymbol{3.277 \times 10^{-8}}$ & 0.09 & 0.9 & 20 & $\boldsymbol{4.715 \times 10^{5}}$ & $\boldsymbol{4.616 \times 10^{-5}}$ \\\hline
 45 & 0.6192 & 0.75 & 4 & 4.3736 & $\boldsymbol{1.733 \times 10^{-11}}$ & 0.09 & 0.9 & 20 & $\boldsymbol{4.816 \times 10^{5}}$ & $\boldsymbol{1.946 \times 10^{-8}}$ \\\hline
\end{tabular}
\caption{Bounds for $\|\pi-\pi_k\|_{TV}$ using (\ref{th-TV-ineg-bis-discret}) /(\ref{rem-spec-n-k-bis-discret})}
\label{Ex_borne_pi}
}
\end{table}
\begin{table}[ht]
\setlength{\extrarowheight}{2pt}
\centering
{\scriptsize 
\begin{tabular}{r||c|c|}  \hline 
  \multicolumn{1}{c}{} & \multicolumn{2}{c|}{$\boldsymbol{(a,b)=(1/2,1/2)}$ with $r_k=0.925$} \\\hline  \hline  
 $\boldsymbol{k}$ & $c_k$ & Bound  (\ref{unif-bound-rate-ter}) \\\hline
 30 & $1.675\times 10^{6}$ & $1.497\times 10^{-7}$ \\\hline
 50 & $6.533\times 10^{6}$ &  $5.839\times 10^{-7}$ \\\hline  
 \multicolumn{3}{c}{(a)}
\end{tabular}
\qquad\qquad \begin{tabular}{rc|c|c|}  \hline 
  \multicolumn{4}{c}{$\boldsymbol{(a,b)=(1/2,1/2)}$ with $r_{35}:=0.9$}   \\ \hline\hline          
 \multicolumn{1}{r|}{$\boldsymbol{\varepsilon}$}  & $10^{-2}$ & $10^{-4}$          &  $10^{-6}$ \\\hline\hline
 \multicolumn{1}{r|}{$\boldsymbol{n(\varepsilon)}$} & 28 & 34 & 40  \\\hline \\
 \multicolumn{4}{c}{(b)}
 \end{tabular}
\caption{(a): bound (\ref{unif-bound-rate-ter}) for $\|P^{300} -\pi(\cdot)1_{\N} \|_1$ and (b): $\|\pi_{n(\varepsilon)} -\pi \|_{TV} \le \varepsilon$ from (\ref{th-TV-ineg-bis-discret})}
\label{Ex_borne_Pn}
}
\end{table}
%

%========================
%=========================
\section{Complements on Theorems~\ref{pro-bv} and \ref{cor-Vk-to-V}} \label{WP_truncation} 
%========================
%=========================

In this section, we assume that $P$ is quasi-compact on $\cB_1$ and that  Conditions~(\ref{C01}) and (\ref{drift-gene-cond-gene}) hold. We use the second part of \cite[Prop.~3.1]{Liv01} and the spectral rank-stability property \cite[Cor.~3.1]{Liv01} to prove the two following properties: 
\leftmargini 1.3em
\begin{itemize}
	\item  if P satisfies (\ref{V-geo-cond}), then $\Pc_k$ is $V$-geometrically ergodic with explicit rates for $k$ large enough. See Proposition~\ref{cor-bv-app}. 
	\item  There exists an integer $\tilde k$ such that, for every $k\geq \tilde k$, the condition $\|P - \Pc_k\|_{0,1} \leq \varepsilon_1(k)$ holds, where $\varepsilon_1(k)$ is defined in (\ref{def-eps-k}). See Proposition~\ref{th-Vk-to-V}. 
\end{itemize} 
As discussed in Sections~\ref{sec-V-to-Vk} and \ref{sec-Vk-to-V}, the previous properties may be valid for $k$ much smaller than the integers $k_0$ and $\tilde k$ provided in this section. It is the reason why these two integers are not introduced in Theorems~\ref{pro-bv} and \ref{cor-Vk-to-V}. However, since the $V_k$-geometric  ergodicity of $P_k$ is required for some $k$ in these two theorems, as well as the condition $\|P - \Pc_k\|_{0,1} \leq \varepsilon_1(k)$ in the second theorem, obtaining the theoretical existence of $k_0$ and $\tilde k$ is important.  

%====================================
%================================================
\subsection{Complements on Theorem~\ref{pro-bv}} \label{app-proof} 
%=================================================
We use the notations introduced for Theorem~\ref{pro-bv}. For any $r\in(0,1)$, set $\eta \equiv \eta(r) :=1-\ln r/\ln \hat\alpha \in(0,1)$, and with  $\varepsilon_1 \equiv \varepsilon_1(r,\vartheta,P)$ defined in~(\ref{def-eps-1})
\begin{gather} 
\varepsilon_2 \equiv \varepsilon_2(r,\vartheta,P) := \bigg\{\frac{r^{n_1}}{4B\big(H (2B+3) + 2(1+B) + (1-r)^{-1}\big)}\bigg\}^{1/\eta}, \nonumber \\ 
\text{and }\varepsilon_0 \equiv \varepsilon_0(r,\vartheta,P) := \min(\varepsilon_1,\varepsilon_2). \label{def-eps-0-app}
\end{gather}
From (\ref{C01}) there exists a positive integer $k_0 \equiv k_0(r,\vartheta,P)$ such that  
\begin{equation} \label{con-k0-r-var}
 \forall k\geq k_0, \quad \|\Pc_k  - P\|_{0,1} \leq \varepsilon_0. 
\end{equation} 
\begin{apro} \label{cor-bv-app} 
Assume that $P$ is $V$-geometrically ergodic  with rate $\rho$ and that Conditions~\emph{(\ref{C01})} and \emph{(\ref{drift-gene-cond-gene})} hold. 
Let $\vartheta>0$ be such that $\max(\hat\alpha,\rho) < 1 - 2\vartheta$ and $r$ such that $\max(\hat\alpha,\rho)+\vartheta < r <1-\vartheta$. Then, for every $k\geq k_0$, Inequality~\emph{(\ref{Vk-geo-cond})} holds with $\rho_k:=r$. 
\end{apro}
Using duality as in the proof of Theorems~\ref{pro-bv}, Proposition~\ref{cor-bv-app} is based on the next lemma which follows from the rank-stability property of spectral projections \cite[cor.~3.1]{Liv01}. 
\begin{alem}  \label{pro-bv-app-bis} 
Assume that $P$ is quasi-compact on $\cB_1$ and that Conditions~\emph{(\ref{C01})} and \emph{(\ref{drift-gene-cond-gene})} hold. 
Let $r>\hat\alpha$ and $\vartheta>0$. For every $k\geq k_0$ the spectral projections of $P'$ and ${\Pc'}_k$ associated with any connected component of $\cV(r,\vartheta,P)$ (not containing $0$) have the same rank. 
\end{alem}
\begin{proof}{ of Proposition~\ref{cor-bv-app}}
From the $V$-geometrical ergodicity of $P$, we have $\sigma(P')=\sigma(P)\subset D(0,\rho) \cup \{1\}$. Thus it follows from (\ref{implic1-proof-dual}) that 
$\sigma(\Pc_k')\subset \overline{D}(0,r)\cup \overline{D}(1,\vartheta)$. Next, from $\|{\Pc_k}  - P\|_{0,1} \leq \varepsilon_0$ and Lemma~\ref{pro-bv-app-bis},  we obtain that $\sigma({\Pc_k}\,')\cap \overline{D}(1,\vartheta) = \{\lambda\}$ for some eigenvalue $\lambda\in\C$ of ${\Pc_k}\,'$ and that there exists an associated rank-one projection $\widehat\Pi_{k,\lambda}'$ on $\cB_1'$ such that we have for any $\kappa\in(r,1-\vartheta)$: 
$$\forall n\geq 1,\quad {\Pc_k}^{\;'n} - \lambda^n \widehat\Pi_{k,\lambda}' = \frac{1}{2i\pi} \oint_{|z|=\kappa} z^n\, (zI-{\Pc_k}\,')^{-1}\, dz.$$
It follows from (\ref{implic1-proof-dual}) that 
\begin{equation} \label{rk-spec3}
\forall n\geq 1,\quad \|{\Pc_k}^{\;'n} - \lambda^n \widehat\Pi_{k,\lambda}'\|_1 \leq c(r,\vartheta,P)\, \kappa^{n+1}. 
\end{equation}
Since $\kappa$ is arbitrarily close to $r$, this gives the expected conclusion in Proposition~\ref{cor-bv-app} using duality and the next lemma. 
\end{proof}
\begin{alem}
The eigenvalue $\lambda$ in \emph{(\ref{rk-spec3})} is equal to 1. Moreover there exists a $\Pc_k$-invariant positive measure $\widehat\pi_k$ on $(\X,\cX)$ such that  $\widehat\pi_k(V) < \infty$, 
and the rank-one projection $\widehat\Pi_{k,\lambda}'$ is the adjoint of the following rank-one projection $\widehat\Pi_k$ on $\cB_1$:   
$$\forall f\in\cB_1,\quad \widehat\Pi_k f := \widehat\pi_k(f)\, \phi_k.$$
\end{alem}
\begin{proof}{}
Since $|\lambda|>\kappa$, we deduce from ${\Pc_k}\phi_k = \phi_k$ and (\ref{rk-spec3}) that, for any $f'\in\cB_1'$, we have $\lim_n \lambda^{-n} f'(\phi_k) = \lim_n \lambda^{-n} ({\Pc_k}^{\;'n} f')(\phi_k) = (\widehat\Pi_{k,\lambda}'f')(\phi_k)$. Since there exists $f'\in\cB_1'$ such that  $f'(\phi_k)\neq 0$,  $\{\lambda^{-n}\}_{n\in\N}$ converges in $\C$. Thus either $\lambda=1$, or $|\lambda| >1$. Moreover the sequence $\{{\Pc_k}^{\;n}\}_{n\in\N}$ is bounded in $\cL(\cB_1)$ (proceed as in (\ref{itere-borne})). Thus $\{{\Pc_k}^{\;'n}\}_{n\in\N}$ is bounded in $\cL(\cB_1')$, so that Inequality~(\ref{rk-spec3}) implies that $\{\lambda^{n}\}_{n\in\N}$ is bounded in $\C$. Therefore $\lambda = 1$.  

Now we omit $\lambda$ in $\widehat\Pi_{k,\lambda}'$. We can prove as in \cite[proof of Th.~1]{FerHerLed11} that $\widehat\Pi_{k}'$ is the adjoint of the  rank-one projection $\widehat\Pi_k$ defined on $\cB_1$ by 
	$$\forall f\in\cB_1,\quad  \widehat\Pi_k f := \lim_n {\Pc_k}^{\;n}f \ \text{ in}\ \cB_1,$$
and that $\widehat\Pi_k$ is of the form $\widehat\Pi_k f = \widehat e_k\, '(f)\, \phi_k$ ($f\in\cB_1$) for some non-negative element $\widehat e_k\, '\in\cB_1'$ and the function $\phi_k$ of the Introduction. Let $x_0\in\X$ such that $\phi_k(x_0)\neq0$. Then we obtain from the previous facts: $\forall A\in\cX,\ \lim_n {\Pc_k}^{\;n}(x_0,A) = \widehat e_k\, '(1_A)\, \phi_k(x_0)$. Since ${\Pc_k}^{\;n}(x_0,\cdot)$ is a positive measure on $(\X,\cX)$, it follows from the Vitali-Hahn-Saks theorem that, for all $A\in\cX$ $\widehat  e_k\, '(1_{A}) = \widehat\pi_k(1_{A})$, where $\widehat\pi_k$ is the positive measure  on $(\X,\cX)$ of Introduction. Proceeding as in \cite{FerHerLed11}, we can prove that $\widehat\pi_k(V) < \infty$ and that $\widehat  e_k\, '$ coincide with $\widehat\pi_k$ on $\cB_1$. 
\end{proof}
%
%
%================================================
\subsection{Complements on Theorem~\ref{cor-Vk-to-V}}  \label{sub-th-Vk-to-V}  
%================================================
%
We use the notations introduced before Theorem~\ref{cor-Vk-to-V}. 
\begin{apro} \label{th-Vk-to-V} 
Assume that $P$ is quasi-compact on $\cB_1$ and that Conditions~\emph{(\ref{C01})} and \emph{(\ref{drift-gene-cond-gene})} hold. Let $\vartheta\in(0,(1-\hat\alpha)/2)$. 
There exists $\tilde k \equiv \tilde k(\vartheta)\in\N^*$ such that, for every $k\in[\tilde k,+\infty)\cap{\cal I}_\vartheta$, the property \emph{(\ref{ineg-p-pk-inv})} holds. 
\end{apro}
This proposition is inspired from \cite[Lem.~4.2]{Liv01}. Since the proof in  \cite{Liv01} is only sketched and the choice of $r_k$ (involved in (\ref{def-eps-k})) must be carefully examined, the derivation of Proposition~\ref{th-Vk-to-V} is detailed in this subsection. First note that, from Definition~\ref{def-q-c} and quasi-compactness of $P$, all the spectral values of $P$ strictly larger than $\hat\alpha$ are eigenvalues since $\hat\alpha \geq r_{ess}(P)$.  
More precisely, for any $R>\hat\alpha$, the operator $P$ has a finite number of eigenvalues of modulus larger than $R$. The same property holds for every $\Pc_k$ since $\Pc_k$ is of finite rank (thus $r_{ess}(\Pc_k)=0$).  

\begin{alem} \label{vp-p-hatp}
Let $R>\hat\alpha$ and $\theta>0$. Then there exists $\widetilde k_0\equiv \widetilde k_0(R,\theta)\in\N^*$ such that, for any eigenvalue of $P$ satisfying $|\lambda| > R$ and for every $k\geq \widetilde k_0$, the open disk $D(\lambda,\theta)$ contains at least an eigenvalue of $\Pc_k$.  
\end{alem}
\begin{proof}{}
 Let $\cD_R$ denote the set of the eigenvalues $z$ of $P$ such that $|z| > R$. Define $a:=R - \hat\alpha$, $b:=\min\{|z-z'|,\, z,z'\in\cD_R, z\neq z'\}$ and $c:=\min\{|z|-R,\, z\in\cD_R\}$. Without loss of generality we can suppose that $\theta< \min(a,b/2,c/2)$. Let  $\vartheta\in(0,\theta)$ and $r := \hat\alpha+\theta$. 
Let $\widetilde k_0$ be the smallest integer such that: $\forall k\geq \widetilde k_0,\ \|\Pc_k  - P\|_{0,1} \leq \varepsilon_0$,   with $\varepsilon_0\equiv\varepsilon_0(r,\vartheta,P)$ given in (\ref{def-eps-0-app}). Note that assumptions of Lemma~\ref{pro-bv-app-bis} are satisfied.

Next, consider any $k\geq \widetilde k_0$ and any $\lambda\in\cD_R$. Let $\theta'\in(\vartheta,\theta)$ be such that $(zI-\Pc_k)^{-1}$ is well-defined in $\cL(\cB_1)$ for every $z\in C(\lambda,\theta') := \{z\in\C : |z-\lambda|=\theta'\}$. Such $\theta'$ exists from the quasi-compactness of $\Pc_k$. Note that $(zI-P)^{-1}$ is also well-defined in $\cL(\cB_1)$ for every $z\in C(\lambda,\theta')$, more precisely: $C(\lambda,\theta')\subset \cV(r,\vartheta,P)^c$. Indeed  let $z\in C(\lambda,\theta')$. Then   
$$|z| \geq |\lambda| - \theta' \geq R + c - \theta' > R + 2\theta- \theta' > R + \vartheta,$$
thus $d(z,\overline{D}(0,R)) > \vartheta$. Next, let $z'\in\sigma(P)$ be such that $|z'|>R$. Then $z'\in\cD_R$. If $z'=\lambda$ then $|z'-z|=|\lambda -z|=\theta' > \vartheta$. If $z'\neq \lambda$ then, using the triangle inequality $|z'-\lambda|\le |z'-z| + |z-\lambda|$, we obtain $|z-z'|\geq b-\theta' > b-\theta >\theta >\vartheta$. We have proved that $d(z,\sigma(P)) > \vartheta$. Finally we have $|z| > r$ since $R>r$ (use $r = \hat\alpha + \theta < \hat\alpha +a = \hat\alpha + R - \hat\alpha = R$). Thus $z\in\cV(r,\vartheta,P)^c$.

Now, the spectral projections 
$$\Pi_{\lambda}' := \frac{1}{2i\pi}\oint_{C(\lambda,\theta')} (zI-P')^{-1}\, dz \quad \text{ and } \quad \Pi_{k,\lambda}' := \frac{1}{2i\pi}\oint_{C(\lambda,\theta')}(zI-\Pc_k')^{-1}\, dz$$
have the same rank from $k\geq \widetilde k_0$ and Lemma~\ref{pro-bv-app-bis}. Since $\Pi_{\lambda}'$ has a nonzero rank from $\lambda\in\sigma(P')$, so is $\Pi_{k,\lambda}'$. Thus we have $D(\lambda,\theta')\cap\sigma(\Pc_k)\neq\emptyset$. 
\end{proof}

Now, to prove  Proposition~\ref{th-Vk-to-V}, we consider any $\vartheta\in(0,(1-\hat\alpha)/2)$ and we set $\tilde{r}:= \hat\alpha + \vartheta/2$. 
\begin{alem} \label{lem-P-Pk}
There exists $\widetilde k_0\equiv \widetilde k_0(\tilde{r},\vartheta)\in\N^*$ such that $\forall k\geq \widetilde k_0,\ \cV(\tilde{r},\vartheta/4,P) \subset \cV(\tilde{r},\vartheta,\Pc_k)$. 
\end{alem}
\begin{proof}{}
Let $u\in \cV(\tilde{r},\vartheta/4,P)$. Thus $|u|\leq \tilde{r}$ or $d(u,\sigma(P))\leq \vartheta/4$. If $|u|\leq \tilde{r}$, then $u\in \cV(\tilde{r},\vartheta,\Pc_k)$. Now assume that $|u|>\tilde{r}$ and $d(u,\sigma(P))\leq \vartheta/4$. Since $\sigma(P)$ is compact, there exists $\lambda\in\sigma(P)$ such that $|u-\lambda|\leq \vartheta/4$. We have 
$|\lambda| > \hat\alpha + \vartheta/4$ from  
$$|\lambda| \geq |u| - \frac{\vartheta}{4} > \tilde{r} - \frac{\vartheta}{4} = \hat\alpha + \frac{\vartheta}{4}.$$
Then it follows from Lemma~\ref{vp-p-hatp}  with $R:=\hat\alpha + \vartheta/4$ and $\theta:=\vartheta/4$ that there exists $\widetilde k_0\equiv \widetilde k_0(R,\theta)\in\N^*$ such that, for every $k\geq \widetilde k_0$, the disk $D(\lambda,\vartheta/4)$ contains an eigenvalue of $\Pc_k$, say $\lambda_k$. We obtain  $d(u,\sigma(\Pc_k)) \leq \vartheta$ since 
$|u - \lambda_k| \leq |u - \lambda| + |\lambda - \lambda_k| \leq \vartheta/2$. Thus $u\in \cV(\tilde{r},\vartheta,\Pc_k)$. 
\end{proof}

From the definition of $\cV(\tilde{r},\vartheta,\Pc_k)$, we have: $z\in\cV(\tilde{r},\vartheta,\Pc_k)^c\, \Rightarrow\, d(z,\sigma(\Pc_k)) > \vartheta$. Thus, the following constant is well-defined for every $k\geq 1$: 
\begin{equation*} 
\widetilde{H}_k := \sup_{z\in \cV(\tilde{r},\vartheta,\Pc_k)^c} \|(zI-\Pc_k)^{-1}\|_1. 
\end{equation*}
\begin{alem} \label{h'k}
The sequence $\{\widetilde{H}_k\}_{k\geq 1}$ is bounded. 
\end{alem}
\begin{proof}{} 
Let $\varepsilon_1(\tilde{r},\vartheta/4,P)$ be defined as in (\ref{def-eps-1}). From (\ref{C01}) there exists $k_1\equiv k_1(\tilde{r},\vartheta)\in\N^*$ such that 
$$\forall k\geq k_1,\quad \|\Pc_k  - P\|_{0,1} \leq \varepsilon_1(\tilde{r},\vartheta/4,P).$$
It follows  from Lemma~\ref{lem-P-Pk} that 
$$\forall k\geq \widetilde k_0,\quad \cV(\tilde{r},\vartheta,\Pc_k)^c \subset \cV(\tilde{r},\vartheta/4,P)^c$$
so that, for every $k\geq  \max(\widetilde k_0,k_1)$ we have from (\ref{th-implic1}) with  $r:=\tilde{r}$ and $\vartheta/4$ in place of $\vartheta$ (note that, from \cite{KelLiv99}, Property~(\ref{th-implic1}) holds for any $r > \hat\alpha$ and $\vartheta>0$ under Conditions~(\ref{C01}), (\ref{drift-gene-cond-gene}) and $r_{ess}(P) < 1$): 
$$\sup_{z\in \cV(\tilde{r},\vartheta,\Pc_k)^c} \|(zI- \Pc_k)^{-1}\|_1 \leq \sup_{z\in \cV(\tilde{r},\vartheta/4,P)^c} \|(zI- \Pc_k)^{-1}\|_1 \ \leq\  c(\tilde{r},\vartheta/4,P)< \infty.$$
This gives the expected assertion. 
\end{proof}
\begin{proof}{ of Proposition~\ref{th-Vk-to-V}} Let $k\in {\cal I}_{\vartheta}$
and $H_k \equiv H(r_k,\vartheta,\Pc_k)$ be defined by (\ref{Itheta}) and (\ref{H-r-vartheta-stat}) respectively. Then $H_k \leq \widetilde{H}_k$. Indeed we have  $\cV(r_k,\vartheta,\Pc_k)^c \subset \cV(\tilde{r},\vartheta,\Pc_k)^c$ since (use (\ref{cor-cond-r-var-k}))
$$\tilde{r}=\hat\alpha + \frac{\vartheta}{2} \leq \max(\hat\alpha,\rho_k) + \vartheta < r_k.$$
It follows from Lemma~\ref{h'k} that $\{H_k\}_{k\geq 1}$ is bounded. 
Then, the sequences $\{n_1(k)\}_{k\ge 1}$  and $\{n_2(k)\}_{k\ge 1}$ given in (\ref{def-n1-n2-k}) are bounded since $\hat\alpha + \vartheta < r_k < 1-\vartheta$. Therefore, the sequence $\{\varepsilon_1(k)\}_{k\ge 1}$  in (\ref{def-eps-k}) is uniformly bounded away from zero, that is $\alpha_1 := \inf_{k\geq 1} \varepsilon_1(k) > 0$. 
Finally, from (\ref{C01}) there exists $\tilde{k}\in\N$ such that: $\forall k\geq \tilde{k},\ \|\Pc_k - P\|_{0,1} \leq \alpha_1$. Thus, for every $k\geq \tilde{k},\ \|\Pc_k - P\|_{0,1} \leq \varepsilon_1(k)$. 
\end{proof} 
\section{Conclusion} 
In all the cases where the probabilistic works cited in Introduction do not provide a satisfactory rate $\rho$ in (\ref{V-geo-cond}), Theorem~\ref{cor-Vk-to-V} can be applied to obtain a new rate $r_k$ and constant $c_k$ in (\ref{V-geo-cond}) derived from $(\rho_k,C_k)$ in (\ref{Vk-geo-cond}) for some $k$ large enough (see (\ref{unif-bound-rate-bis})). This new pair $(r_k,c_k)$ will be all the more interesting that 
\leftmargini 1.5em
\begin{enumerate}
	\item the pair $(\rho_k,C_k)$ in (\ref{Vk-geo-cond}) is precise; 
	\item the bound on the essential spectral radius $r_{ess}(P)$  is accurate. 
\end{enumerate}
The  estimate $\|\widehat\pi_n - \pi \|_{TV} = O(\Delta_n|\ln \Delta_n|)$ provided by (\ref{th-TV-ineg-bis}) then involves a computable constant depending on $(r_k,c_k)$. As mentioned in Remark~\ref{effectiv-var-total} and illustrated in Section~\ref{sec-appli}, the direct Inequality~(\ref{rem-spec-n-k-bis}) provides the best estimate of the total variation distance between $\pi$ and $\widehat\pi_k$ when the $C_k$'s are accurately computed.

The point~1.~above is of computational nature (see Section~\ref{sec-appli} for discrete $\X$). The point~2.~is addressed in Theorem~\ref{main} and Proposition~\ref{pro-qc-bis}. Note that $r_{ess}(P)$  is less than the contractive coefficient $\delta$ of (\ref{drift-gene-cond-gene-P}) in  Proposition~\ref{pro-qc-bis} and in atomic case (see Remark~\ref{rk-atom}). It would be of interest to know whether inequality $r_{ess}(P)\leq\delta$ extends to other situations, and more generally whether the general bound in Theorem~\ref{main} can be improved. 

\paragraph*{Acknowledgment} 
We wish to thank the referee for her/his useful advices and encouragements which allowed us to improve our article. We also thank Bernard Delyon for stimulating discussions concerning the results of Section~\ref{sec-mino}.

%==============================
%===============================
\bibliographystyle{plain}

\end{document}